\documentclass[11pt]{amsart}

\usepackage{amsopn, amssymb, amscd, amsmath, amsthm, MnSymbol, stmaryrd}
\usepackage{enumerate}
\usepackage{color}
\usepackage [all] {xy}
\usepackage[perpage,symbol*]{footmisc}
\usepackage{array}
\usepackage{multirow}
\usepackage{graphicx}
\usepackage{color, colortbl}
\usepackage{array, booktabs, arydshln, xcolor}
\usepackage{dsfont}
\usepackage{mathrsfs}

\usepackage{anysize}
\marginsize{2.5cm}{2.5cm}{3cm}{3cm}
\newcommand{\nc}{\newcommand}

\nc{\gramU}{\textbf{\textsl{U}}}

\nc{\fg}{\mathfrak{f} }     \nc{\vg}{\mathfrak{v} }       \nc{\wg}{\mathfrak{w} }
\nc{\zg}{\mathfrak{z} }     \nc{\ngo}{\mathfrak{n} }      \nc{\kg}{\mathfrak{k} }
\nc{\ngoc}{\widehat{\mathfrak{n}} }
\nc{\mg}{\mathfrak{m} }     \nc{\bg}{\mathfrak{b} }       \nc{\ggo}{\mathfrak{g} }
\nc{\ggoc}{\widehat{\mathfrak{g}} }
\nc{\sog}{\mathfrak{so} }
\nc{\sug}{\mathfrak{su} }   \nc{\spg}{\mathfrak{sp} }     \nc{\slg}{\mathfrak{sl} }
\nc{\glr}{\mathfrak{gl}_{n}(\RR)}   \nc{\cg}{\mathfrak{c} }       \nc{\rg}{\mathfrak{r} }
\nc{\g}{\mathfrak{gl}_n(\RR)}
\nc{\glg}{\mathfrak{gl}}
\nc{\hg}{\mathfrak{h} }     \nc{\tg}{\mathfrak{t} }       \nc{\ug}{\mathfrak{u} }
\nc{\dg}{\mathfrak{d} }     \nc{\ag}{\mathfrak{a} }       \nc{\pg}{\mathfrak{p} }
\nc{\agc}{\widehat{\mathfrak{a}} }
\nc{\sg}{\mathfrak{s} }     \nc{\affg}{\mathfrak{aff} }
\nc{\ggob}{\overline{\mathfrak{g}} }

\nc{\pca}{\mathcal{P}}       \nc{\nca}{\mathcal{N}}       \nc{\lca}{\mathcal{L}}
\nc{\oca}{\mathcal{O}}       \nc{\mca}{\mathcal{M}}       \nc{\tca}{\mathcal{T}}
\nc{\aca}{\mathcal{A}}       \nc{\cca}{\mathcal{C}}       \nc{\gca}{\mathcal{G}}
\nc{\sca}{\mathcal{S}}       \nc{\hca}{\mathcal{H}}       \nc{\bca}{\mathcal{B}}
\nc{\dca}{\mathcal{D}}       \nc{\rca}{\mathcal{R}}

\nc{\val}{\operatorname{val}}

\nc{\vp}{\varphi}
\nc{\ddt}{\tfrac{{\rm d}}{{\rm d}t}}
\nc{\dpar}{\tfrac{\partial}{\partial t}}

\nc{\im}{\sqrt{-1}}        

\newcommand{\fnn}[3]{
\begin{array}{rccl}
{#1}:&\hspace{-2mm}{#2}&\hspace{-2mm}\longrightarrow&\hspace{-2mm}{#3}\\
\end{array}}

\newcommand{\comillas}[1]{\textquotedblleft{#1}\textquotedblright}

\nc{\SO}{\mathrm{SO}}           \nc{\Spe}{\mathrm{Sp}}          \nc{\Sl}{\mathrm{SL}}
\nc{\SU}{\mathrm{SU}}           \nc{\Or}{\mathrm{O}(n)}            \nc{\U}{\mathrm{U}}
\nc{\Se}{\mathrm{S}}            \nc{\Cl}{\mathrm{Cl}}           \nc{\Spein}{\mathrm{Spin}}
\nc{\Pin}{\mathrm{Pin}}
\nc{\Glr}{\mathrm{GL}_n(\RR)}   \nc{\Glc}{\mathrm{GL}_n(\CC)}   \nc{\Glv}{\mathrm{GL}(V)}    \nc{\Glk}{\mathrm{GL}_n(\fk)}   \nc{\Gl}{\mathrm{GL}}
\nc{\GrpG}{\mathrm{G}}          \nc{\GrpH}{\mathrm{H}}          \nc{\GrpA}{\mathrm{A}}       \nc{\GrpT}{\mathrm{T}}          \nc{\GrpK}{\mathrm{K}}
\nc{\GrpGc}{\widehat{\mathrm{G}}}
\nc{\GrpN}{\mathrm{N}}

\nc{\RR}{{\Bbb R}} \nc{\HH}{{\Bbb H}} \nc{\CC}{{\Bbb C}} \nc{\ZZ}{{\Bbb Z}}
\nc{\FF}{{\Bbb F}} \nc{\NN}{{\Bbb N}} \nc{\QQ}{{\Bbb Q}} \nc{\PP}{{\Bbb P}}

\nc{\euler}{{\rm e}}

\nc{\vs}{\vspace{.2cm}} \nc{\vsp}{\vspace{1cm}}

\nc{\ip}{\langle \cdot , \cdot \rangle}
\nc{\ipd}{\langle \hspace{-0.5mm}\langle \cdot , \cdot \rangle\hspace{-0.5mm}\rangle}
\nc{\ippd}{( \hspace{-0.5mm} ( \cdot , \cdot ) \hspace{-0.5mm} )}
\nc{\ipp}{(      \cdot , \cdot       )}
\nc{\la}{\langle} \nc{\ra}{\rangle}

\nc{\ortsum}{ \mbox{\tiny $\displaystyle \bigoplus^{\perp}$}}

\nc{\dirsum}{ \mbox{\tiny $\displaystyle \bigoplus $}}

\nc{\ortres}{ \mbox{\tiny $\displaystyle \bigominus^{\perp}$}}

\newcommand{\ipa}[2]{\langle {#1} , {#2} \rangle}

\newcommand{\ipda}[2]{\langle \hspace{-0.5mm}\langle {#1} , {#2} \rangle\hspace{-0.5mm}\rangle}

\newcommand{\ippa}[2]{(#1, #2)}

\nc{\unm}{\tfrac{1}{2}}\nc{\unc}{\tfrac{1}{4}} \nc{\und}{\tfrac{1}{16}}

\nc{\no}{\vs\noindent}

\nc{\lamn}{\Lambda^2(\RR^n)^*\otimes\RR^n} \nc{\lamp}{\Lambda^2\pg^*\otimes\pg}
\nc{\lamg}{\Lambda^2\ggo^*\otimes\ggo} \nc{\lamngo}{\Lambda^2\ngo^*\otimes\ngo}
\nc{\lamnk}{\Lambda^2(\fk^n)^*\otimes \fk^n} \nc{\lamnkt}{\Lambda^3(\fk^n)^*\otimes \fk^n}

\nc{\tangz}{{\rm T}^{\rm Zar}}
\nc{\mum}{/\!\!/} \nc{\kir}{/\!\!/\!\!/}
\nc{\lievark}{\mathfrak{L}_n(\fk)}         \nc{\lievarc}{\mathfrak{L}_n(\CC)}        \nc{\lievarr}{\mathfrak{L}_n(\RR)}
\nc{\solvvark}{\mathfrak{R}_n(\fk)}        \nc{\solvvarc}{\mathfrak{R}_n(\CC)}       \nc{\solvvarr}{\mathfrak{R}_n(\RR)}
\nc{\nilvark}{\mathfrak{N}_n(\fk)}         \nc{\nilvarc}{\mathfrak{N}_n(\CC)}        \nc{\nilvarr}{\mathfrak{N}_n(\RR)}
\nc{\cirre}{\textrm{C}}
\nc{\fk}{\mathrm{k}}

\nc{\Ri}{\tfrac{4\Ric_{\mu}}{||\mu||^2}}

\nc{\ds}{\displaystyle}

\nc{\lb}{[\cdot,\cdot]}

\nc{\Hess}{\operatorname{Hess}}
\nc{\diag}{\operatorname{Diag}}   \nc{\Id}{\operatorname{Id}}
\nc{\trans}{\mbox{{\tiny$\operatorname{T}$}}}                         \nc{\Proj}{\operatorname{Proj}}
\nc{\Proy}{\operatorname{Proy}}
\nc{\ad}{\operatorname{ad}}       \nc{\Ad}{\operatorname{Ad}}        %\nc{\exp}{\operatorname{exp}}
\nc{\rank}{\operatorname{rank}}   \nc{\codim}{\operatorname{codim}}  %dim \dim
\nc{\Irr}{\operatorname{Irr}}     \nc{\End}{\operatorname{End}}
\nc{\Aut}{\operatorname{Aut}}     \nc{\Inn}{\operatorname{Inn}}
\nc{\lRad}{\operatorname{Rad}}
\nc{\Der}{\operatorname{Der}}     \nc{\Ker}{\operatorname{Ker}}
\nc{\Iso}{\operatorname{I}}       \nc{\Diff}{\operatorname{Diff}}
\nc{\Lie}{\operatorname{Lie}}     \nc{\tr}{\operatorname{tr}}
\nc{\dif}{\operatorname{d}}       \nc{\e}{\operatorname{e}}
\nc{\sen}{\operatorname{sen}}     \nc{\tang}{\operatorname{T}}
\nc{\modu}{\operatorname{mod}}
\nc{\Riem}{\operatorname{Rm}}     \nc{\Ric}{\operatorname{Ric}}
\nc{\chP}{\operatorname{P}}       \nc{\chPform}{\operatorname{p}}
\nc{\sym}{\operatorname{sym}}     \nc{\symac}{\operatorname{sym^{ac}}}   \nc{\symc}{\operatorname{sym^{c}}}
\nc{\scalar}{\operatorname{sc}}
\nc{\grad}{\operatorname{grad}}
\nc{\ricci}{\operatorname{ric}}   \nc{\nr}{\operatorname{nr}}            \nc{\riccic}{\operatorname{ric^{c}}}
\nc{\riccig}{\operatorname{ric^{\gamma}}}
\nc{\Rin}{\operatorname{M}}
\nc{\Kill}{\operatorname{B}}
\nc{\Le}{\operatorname{L}}
\nc{\level}{\operatorname{level}} \nc{\rad}{\operatorname{r}}
\nc{\abel}{\operatorname{ab}}
\nc{\CH}{\operatorname{CH}}        \nc{\mcc}{\operatorname{mcc}}     \nc{\inte}{\operatorname{int}}         \nc{\aff}{\operatorname{Aff}}
\nc{\CaC}{\operatorname{CC}}        \nc{\ccm}{\operatorname{ccm}}
\nc{\Adj}{\operatorname{Adj}}
\nc{\Order}{\operatorname{O}} \nc{\Ricg}{\operatorname{Ric^{\gamma}}}
\nc{\Hom}{\operatorname{Hom}}
\nc{\sign}{\operatorname{sign}}
\nc{\spanv}{\operatorname{span}}
\nc{\xp}{\operatorname{xp}}   \nc{\xt}{\operatorname{xt}}
\nc{\IC}{\operatorname{IC}}   \nc{\OC}{\operatorname{OC}}
\nc{\signo}{\operatorname{sgn}}

\nc{\rhov}{\operatorname{\rho_{v}}}
\nc{\mm}{m}
\nc{\mmt}{\widetilde{m}}
\nc{\F}{\operatorname{F}}

\newtheorem{theoremABC}{Theorem}

\theoremstyle{plain}
\newtheorem{theorem}{Theorem}[section]
\newtheorem{proposition}[theorem]{Proposition}

\theoremstyle{definition}
\newtheorem{definition}[theorem]{Definition}
\newtheorem{notation}[theorem]{Notation}

\theoremstyle{remark}
\newtheorem{remark}[theorem]{Remark}

\newtheorem{example}[theorem]{Example}

\numberwithin{equation}{section}

\begin{document}

\title[Soliton almost K\"{a}hler structures on 6-dimensional nilmanifolds ]{Soliton almost K\"{a}hler structures on 6-dimensional nilmanifolds for the symplectic curvature flow}
\author{EDISON ALBERTO FERN\'ANDEZ-CULMA}
\address{Current affiliation: CIEM, FaMAF, Universidad Nacional de C\'ordoba, \newline \indent Ciudad Universitaria, \newline \indent (5000) C\'ordoba, \newline \indent Argentina}
\email{efernandez@famaf.unc.edu.ar}
\thanks{Fully supported by a CONICET Postdoctoral Fellowship (Argentina)}

\begin{abstract}
The aim of this paper is to study self-similar solutions to the symplectic cuvature flow on $6$-dimensional nilmanifolds.
For this purpose, we focus our attention in the family of symplectic Two- and Three-step nilpotent Lie algebras admitting
a \textit{minimal compatible metric} and we give a complete classification of these algebras together with their respective metric.
Such classification is given by using our generalization of Nikolayevsky's nice basis criterium,
which will be repeated here in the context of canonical compatible metrics for geometric structures on nilmanifolds,
for the convenience of the reader.

By computing the Chern-Ricci operator $\chP$ in each case, we show that the above distinguished metrics define
a soliton almost K\"{a}hler structure.

Many illustrative examples are carefully developed.

\end{abstract}

\maketitle

\section{Introduction}

Let $(g,J,\omega)$ be an almost K\"{a}hler structure on a manifold $M^{2n}$. Let us denote by $\chPform$ its Chern-Ricci form and
by $\ricci$ the usual Ricci tensor of the Riemannian manifold $(M^{2n},g)$. The \textit{symplectic curvature flow} on a
compact almost K\"{a}hler manifold $(M^{2n},g_{0},J_{0},\omega_{0})$ is given by the system of evolution equations
\begin{eqnarray}\label{SCF}
\left\{\begin{array}{ll}
  \dpar \omega = -2\chPform;                            &\mbox {with } \omega(0)=\omega_{0},\\
  \dpar g= -2(\chPform^{c}(\cdot,J\cdot)-\ricci^{ac} ); &\mbox{with } g(0)=g_{0}.
\end{array}\right.
\end{eqnarray}
Here, $\chPform^{c}$ is the complexified component of $\chPform$ (also called $J$-invariant part of $\chPform$) and
$\ricci^{ac}$ denotes the anti-complexified part of $\ricci$ (also known as anti-$J$-invariant part of $\ricci$).

This geometric flow was recently introduced by Jeffrey Streets and Gang Tian in \cite{STREETS1}, where it is proved the short-time existence and
uniqueness for this flow. The solution to the Equation (\ref{SCF}) preserves the almost K\"{a}hler structure, and if the initial almost K\"{a}hler structure is in fact K\"{a}hler, then such solution is a solution to the K\"{a}hler Ricci flow.

Let $\GrpG$ be a simply connected Lie group admitting a left-invariant almost K\"{a}hler structure $(g_{0},J_{0},\omega_{0})$. Any left-invariant almost K\"{a}hler structure on $\GrpG$ is determined by a inner product $\ip$ on $\ggo$ and a non-degenerate skew-symmetric bilinear form $\omega$ on $\ggo$, here $\ggo=\Lie(\GrpG)$ (the Lie algebra of $\GrpG$). One can consider the symplectic curvature flow (SCF for short) on $\GrpG$, where (\ref{SCF}) becomes a system of ordinary differential equations. The almost K\"{a}hler structure $(g_{0},J_{0},\omega_{0})$ is called a \textit{soliton} (\cite[Section 7]{LAURET7}), if the solution to the SCF starting at $(g_{0},J_{0},\omega_{0})$ is (algebraically) self-similar, it is to say,
the solution has the form:
\begin{eqnarray}\label{solitondef1}
\left\{\begin{array}{ll}
\omega_{t} = c(t)\omega_{0}(\phi_{t}\cdot,\phi_{t}\cdot)\\
g_{t} = c(t) g_{0}(\phi_{t}\cdot , \phi_{t} \cdot)
\end{array}\right.
\end{eqnarray}
for some $c(t) \in \RR_{>0}$, $\phi_{t} \in \Aut(\ggo)$, both differentiable at $t$, with $c(0)=1$, $\phi_{0}=\Id$ and $\phi^{\prime}_{0}= D \in \Der(\ggo)$.

By following results given in \cite{LAURET7}, our aim in this note is to
study soliton K\"{a}hler structures on $6$-dimensional nilmanifolds. In some cases, such structures are determined by \textit{minimal compatible metrics} on symplectic nilpotent Lie algebras. This is the case of symplectic two-step nilpotent Lie algebras which are Chern-Ricci flat (it follows from results of Luigi Vezzoni in \cite{VEZZONI1} or \cite[Proposition 2]{POOK1}). In this way, we give a complete classification of minimal compatible metrics with symplectic three-step and two-step nilpotent Lie algebras and prove that:

\begin{theoremABC}\label{th1}
All symplectic two-step Lie algebras of dimension $6$ admit a minimal compatible metric, and, in consequence, admit a soliton almost K\"{a}hler structure.
\end{theoremABC}

\begin{theoremABC}\label{th2}
Every minimal compatible metric on a symplectic three-step nilpotent Lie algebra of dimension 6 defines a soliton almost K\"{a}hler structure.
\end{theoremABC}

In general, it is a difficult problem to know when a symplectic nilpotent
Lie algebra admits a minimal compatible metric. Such problem is equivalent to determining if a orbit of the natural action of $\Spe(n,\RR)$ on $\Lambda^2(\RR^{2n})^*\otimes\RR^{2n}$ is \textit{distinguished}, i.e. we must determine when a orbit contains a critical point of the norm-square of the moment map $\mm_{\spg}$ associated to the action. In \cite{FERNANDEZ-CULMA2}, by using convexity properties of the moment map and recent results of Michael Jablonski in \cite{JABLONSKI2}, we gave a criterion to know when a \textit{nice element} of a real reductive representation has a distinguished orbit. Such result can be considered as a generalization of Nikolayevsky's nice basis criterium (see \cite[Theorem 3]{NIKOLAYEVSKY2}). Theorems \ref{th1} and \ref{th2} are proved by using such criterion, and so, for the convenience of the reader, we give the corresponding criterion in the context of canonical compatible metrics for geometric structures on nilmanifolds (subsection \ref{canonicalmetric}).

\section {Preliminaries}

\subsection{Soliton solutions for the SCF on Lie groups}

Let $\GrpG$ be a Lie group admitting an left-invariant almost K\"{a}hler structure; i.e there exist a symplectic structure $\omega$ on $\ggo$, an almost complex structure $J$ on $\ggo$ and a inner product $\ip$ on $\ggo$ satisfying the \textit{compatibility condition}
\begin{eqnarray*}
  \omega(X,Y)=\ipa{JX}{Y}, \, \forall X,Y \in \ggo.
\end{eqnarray*}
By symplectic structure on $\ggo$ we mean that the non-degenerate skew-symmetric bilinear form $\omega$ is \textit{closed}:
\begin{eqnarray}\label{condsymplectic}
  \omega([X,Y],Z)+\omega([Y,Z],X)+\omega([Z,X],Y)=0
\end{eqnarray}
for any $X$, $Y$ and $Z$ in $\ggo$.

Given a bilinear form on $\ggo$, say $\fnn{B}{\ggo \times \ggo}{\RR}$, the complexified part of $B$ and the anti-complexified part of $B$, denoted by $B^{c}$ and $B^{ac}$ respectively, are defined to be
\begin{eqnarray*}
B^{c}(\cdot,\cdot)=\frac{1}{2}(B(\cdot,\cdot)+B(J\cdot,J\cdot)) \mbox{ and }  B^{ac}(\cdot,\cdot)=\frac{1}{2}(B(\cdot,\cdot)-B(J\cdot,J\cdot)).
\end{eqnarray*}In the same manner are defined the complexified part and anti-complexified part of a linear map $\fnn{T}{\ggo}{\ggo}$, denoted by $T^{c}$ and $T^{ac}$, to be
\begin{eqnarray*}
T^{c}=\frac{1}{2}(T-JTJ) \mbox{ and } T^{ac}=\frac{1}{2}(T+JTJ)
\end{eqnarray*}

From now on, the transpose of a linear map $\fnn{T}{\ggo}{\ggo}$ with respect to $\ip$ and $\omega$ are  denoted by $T^{\trans}$ and $T^{{\trans}_{\omega}}$ respectively (Note that $T^{{\trans}_{\omega}} = -JA^{\trans}J$).

Let $\chPform$ be the Chern-Ricci form of the left-invariant almost K\"{a}hler structure $(\ip,J,\omega)$. It is proved in \cite[Proposition 4.1]{VEZZONI1} and \cite{POOK1} that $\chPform$ is given by
\begin{eqnarray}\label{CRform}
  \chPform(X,Y) = \frac{1}{2}\left( \tr(\ad_{J[X,Y]}) - \tr(J ad_{[X,Y]} ) \right).
\end{eqnarray}

The above expression has interesting consequences, among them let us mention, for instance, if $\ggo$ is a two-step nilpotent Lie algebra ($\ad_{[\cdot,\cdot]}=0$), then $\chPform = 0$; i.e. $(\ggo,\ip,J,\omega)$ is Chern-Ricci flat (\cite[Proposition 4.1]{VEZZONI1} or \cite[Proposition 2]{POOK1}).

From Equation (\ref{CRform}), it is easy to see that there exists an $\widehat{H}\in \ggo$ such that
\begin{eqnarray}
  \chPform(X,Y) = \omega(\widehat{H},[X,Y])
\end{eqnarray}
because $\omega$ is non-degenerate. Such $\widehat{H}$ can be taken to be $\frac{1}{2} \sum \ad_{e_i}^{\trans}e_i + \frac{1}{2}\sum J\ad_{e_i}^{\trans}(Je_i)$
where $\{e_1,\ldots,e_n\}$ is a orthonormal basis of $\ggo$. Therefore, if $\chP$ is the Chern-Ricci operator; i.e. $\chP$ is the linear transformation of $\ggo$ such that $\chPform(X,Y)=\omega(\chP X,Y)$ for all $X,Y \in \ggo$, then $\chP = \ad_{\widehat{H}} + \ad_{\widehat{H}}^{{\trans}_{\omega}}$ (it is immediate from Equation (\ref{condsymplectic})).

According to the above formula for $\chP$, we can say more about the Chern-Ricci form $\chPform$. To do this, let us make a short digression on left-symmetric algebras.

\begin{definition}\cite{SEGAL}
A \textit{left-symmetric algebra structure} (or LSA-structure for short) on a Lie algebra $\ggo$ is a bilinear product $\fnn{\centerdot}{\ggo \times \ggo}{\ggo}$ satisfying the conditions
\begin{enumerate}
  \item $[X,Y] = X\centerdot Y - Y \centerdot X$,
  \item $X\centerdot(Y\centerdot Z)-Y\centerdot(X\centerdot Z)-[X,Y]\centerdot Z = 0$
\end{enumerate}
for all $X$, $Y$ and $Z \in \ggo$.

Given $X \in \ggo$, let $\lambda_{X}$ denote (respectively $\rho_{X}$) the left multiplication by $X$ (respectively right multiplication by $X$) in the left-symmetric algebra: $\lambda_{X}Y=X\centerdot Y$ and $\rho_X(Y)=Y\centerdot X$ for all $Y \in \ggo$. The LSA-structure is called \textit{complete} if for every $X \in \ggo$, the linear transformation $\Id + \rho_{X}$ is bijective.
\end{definition}

Note that LSA-structure conditions are equivalent to having a left-invariant affine connection which is: (1) torsion free and (2) flat. These concepts play an important role in the study of affine crystallographic groups and of fundamental groups of affine manifolds, which are well-developed theories and have a rich history that includes challenging problems due to Louis Auslander and John W. Milnor. We refer the reader to \cite{BURDE1} for a comprehensive review of the literature on such topics.

On completeness of a LSA-structure, we have the following result due to Dan Segal.

\begin{theorem}\cite{SEGAL}
Let $\ggo$ be a Lie algebra over a field $\fk$ of characteristic zero and $\fnn{\centerdot}{\ggo \times \ggo}{\ggo}$ be a LSA-structure on $\ggo$.
The the following conditions are equivalent:
\begin{enumerate}
  \item The LSA-structure is complete.
  \item The LSA-structure is right nilpotent, i.e., $\rho_{X}$ is a nilpotent linear transformation, for all $X \in \ggo$.
  \item $\tr(\rho_{X})=0$ for all $X \in \ggo$.
\end{enumerate}
\end{theorem}

From the above theorem, we have a additional property on the Chern-Ricci operator

\begin{proposition}
Let $\ggo$ be a unimodular Lie algebra and $(\ip,J,\omega)$ be an almost-K\"{a}hler
structure on $\ggo$. Then its Chern-Ricci operator $\chP$ is a nilpotent operator.
\end{proposition}

\begin{proof}
Consider the usual LSA-structure $\fnn{\centerdot}{\ggo \times \ggo}{\ggo}$ on $\ggo$ induced by the symplectic structure,
which is defined implicitly by
\begin{eqnarray}\label{eq1}
  \omega(H,[X,Y])=-\omega(X\centerdot H,Y),
\end{eqnarray}
for any $H$, $X$ and $Y \in \ggo$.

Given $H \in \ggo$, let $P_{H}$ denote the linear transformation such that
\begin{eqnarray}\label{eq2}
  \omega(H,[X,Y])= \omega(P_{H}X,Y), \forall X,Y \in \ggo.
\end{eqnarray}
From Equation (\ref{condsymplectic}), it follows that
\begin{eqnarray*}
  \ad_{X}^{\trans_{\omega}}(H)&=& \ad_{H}^{\trans_{\omega}}(X)+ \ad_{H}(X),
\end{eqnarray*}
and, in consequence, $P_{H} = \ad_{H}^{\trans_{\omega}} + \ad_{H}$. Since
$\ggo$ is unimodular ($\tr(\ad_{Z})=0, \, \forall Z\in \ggo$), $\tr(P_{H})=2\tr(\ad_{H})=0$,
for any $H\in \ggo$.

By Equations (\ref{eq1}) and (\ref{eq2}), we have $\rho_{H}=-P_{H}$, and, on account of the above-mentioned,
$\tr(\rho_{H})=0$ for any $H\in \ggo$, which implies that the LSA-structure is complete, and so,
$\rho_{H}$ is a nilpotent linear transformation for any $H\in\ggo$. Since the Chern-Ricci operator
$\chP$ is equal to $P_{\widehat{H}}=-\rho_{\widehat{H}}$ to certain $\widehat{H} \in \ggo$, it completes the proof.
\end{proof}

\begin{remark}
 A direct proof of the above proposition can be given by using the formula:
\begin{eqnarray}
\tr(\rho_{X}^{k}) &=& \tr(\rho_{X^{k}}), \,  \forall X \in \ggo
\end{eqnarray}
where $X^{k}=X^{k-1}\centerdot X$ with $k \in \NN$ (see \cite[Proposition 15]{HELMSTETTER1}, \cite[Theorem 2.2]{KIM1} or \cite[Proposition 2]{SEGAL})
\end{remark}

Having disposed of this preliminary information on the Chern-Ricci form,
we can now return to our objective.

\begin{definition}\cite[Definition 7.2]{LAURET7}\label{solitondef2}
An almost K\"{a}hler structure $(\ip,J,\omega)$ on a Lie algebra $\ggo$ is called \textit{soliton} if for some $c\in\RR$ and $D \in \Der(\ggo)$,
\begin{eqnarray}
  \left\{\begin{array}{l}
    \chP = c\Id + \frac{1}{2}(D-JD^{\trans}D)\\
    \chP^{c} + \Ric^{ac} = c\Id + \frac{1}{2}(D + D^{\trans})
 \end{array}\right.
\end{eqnarray}
\end{definition}

It is proved in \cite[Lemma 7.1]{LAURET7} that the above definition is equivalent to that the solution to the SCF starting in $(\ip,J,\omega)$ is
self-similar in the sense of la Condition (\ref{solitondef1}).

Note that if the almost K\"{a}hler structure is in fact K\"{a}hler, then such structure is a soliton if and only if $\ip$ is a semi-algebraic algebraic Ricci soliton; $\Ric=c\Id + \frac{1}{2}(D + D^{\trans})$ (because $\chP=\Ric$).

Given that the SCF evolves the metric and the symplectic structure, preserving the compatibility, one expects that, in general, it is not enough to have a \comillas{distinguished} metric or a \comillas{distinguished} symplectic structure in order to have a soliton.

\begin{example}
Consider the Lie algebra $\ggo:=(\RR^{6},\mu)$ with
\begin{eqnarray*}
  \mu&=& \left\{ \begin{array}{l}
                [{e_1},{e_2}]={e_1},[{e_1},{e_3}]={e_1},[{e_1},{e_4}]=-2\,{e_6},[{e_1},{e_6}]=-2\,{e_5},\\
                {[{e_2},{e_5}]}=-2\,{e_5},[{e_2},{e_6}]=-{e_6},[{e_3},{e_4}]=2\,{e_4},[{e_3},{e_6}]={e_6} \end{array}\right.
\end{eqnarray*}
and the almost-K\"{a}hler structure on $\ggo$ given by $(\ip,J,\omega_{cn})$ where $\ip$ is the
canonical inner product of $\RR^6$ and $\omega_{cn}$ is the canonical symplectic form of $\RR^6$;
$\omega_{cn}=e_{1}^{\ast}\wedge e_{6}^{\ast} + e_{2}^{\ast}\wedge e_{5}^{\ast} + e_{3}^{\ast}\wedge e_{4}^{\ast}$.

An easy computation shows that $(\ip,J,\omega_{cn})$ is an Einstein strictly almost K\"{a}hler structure on $\ggo$. In fact, $N_{J}(e_1,e_3)=4e_1 \neq 0$ where $N_{J}$ is the Nijenhuis tensor and
\begin{eqnarray*}
  \Ric &=& \diag(-4, -3, -3, -2, 2, 0) -\frac{1}{2}(0,6,6,0,0,0)-\diag(2,0,0,4,8,6)\\
       &=& -6\Id.
\end{eqnarray*}
In consequence, the Ricci tensor is $J$-invariant; $\Ric^{ac}=0$.

Let $\widehat{H} = 3e_2-e_3$, which is such that $\chPform(X,Y)=\omega_{cn}(\widehat{H},[X,Y])$, thus the Chern-Ricci operator $\chP$ is given by
\begin{eqnarray*}
  \chP &=& \diag(-6,-6,-2,-2,-6,-6).
\end{eqnarray*}
In consequence, the Chern-Ricci operator is symmetric; $\chP^{c}=\chP$.

In this case, the soliton condition may then be reduced to $\chP = c\Id + \frac{1}{2}(D + D^{\trans})$, with $D \in \Der(\ggo)$. Since the algebra of derivations of $\ggo$ is given by
\begin{eqnarray*}
\Der(\ggo) &=& \spanv \left\{ \begin{array}{l}
E_{1,1}+2E_{5,5}+E_{6,6},E_{4,4}+E_{5,5}+E_{6,6},E_{6,1}-E_{4,3},\\
E_{5,1}+\frac{1}{2}(E_{6,2}-E_{6,3}),(E_{1,2}+E_{1,3})-2(E_{6,4}+E_{5,6},E_{5,2})
\end{array} \right\},
\end{eqnarray*}
a trivial verification shows that $(\ip,J,\omega)$ is not a soliton.
\end{example}

\begin{example}
Consider the family of strictly almost K\"{a}hler solvmanifolds given by the family of solvable Lie algebras $\ggo(\lambda_{1},\lambda_{2},\lambda_{3}):=(\RR^6,\mu_{\lambda_{1},\lambda_{2},\lambda_{2}})$
where
\begin{eqnarray*}
\mu_{\lambda_{1},\lambda_{2},\lambda_{2}} &=&\left\{\begin{array}{l}
[{e_1},{e_3}]= -\left( {\lambda_{{1}}}^{2}+{\lambda_{{2}}}^{2} \right) {e_3},
[{e_1},{e_4}]=\lambda_{{3}}{e_3}+
\left( {\lambda_{{1}}}^{2}+{\lambda_{{2}}}^{2} \right) {e_4},\\
{[{e_2},{e_6}]}= \left({\lambda_{{2}}}^{2}  -{\lambda_{{1}}}^{2}\right) {e_2}+2\,\lambda_{{1}}\lambda_{{2}}{e_5},\\
{[{e_5},{e_6}]}=2\,\lambda_{{1}}\lambda_{{2}}{e_2}+ \left( {\lambda_{{1}}}^{2}-{\lambda_{{2}}}^{2} \right) {e_5},
\end{array}\right.
\end{eqnarray*}
with strictly almost K\"{a}hler structure $(\ip,J,\omega_{cn})$ given by the canonical inner product of $\ip$ and the usual symplectic form $\omega_{cn}$ (see that $N_{J}(e_1,e_4)=\lambda_{3}e_3+2(\lambda_{1}^{2}+\lambda_{2}^{2})e$).

Let $\widehat{H}=\frac{1}{2}\lambda_{3}e_{6}$ which is such that $\chPform(X,Y)=\omega_{cn}(\widehat{H},[X,Y])$. Since $e_6 \perp^{\omega}[\ggo,\ggo]$, the above family is Chern-Ricci flat ($\chPform=0$).

An easy computation shows that
\begin{eqnarray*}
  \Ric^{ac}&=& \frac{\lambda_{3}}{4}\diag\left(-\lambda_{3},0,\left(\begin{array}{cc}2\lambda_{3} & 4(\lambda_{1}^{2}+\lambda_{2}^2)\\ 4(\lambda_{1}^{2}+\lambda_{2}^2) & -2\lambda_{3}\end{array}\right),0,\lambda_{3}\right)
\end{eqnarray*}
and a straightforward computation of the algebra of derivations when $\lambda_{3}\neq0$ shows that the above structure is a soliton if and only if $\lambda_{3}=0$, and in such case, $\Ric^{ac}=0$.
\end{example}

Some sufficient conditions to have a soliton has been given in \cite{LAURET7}. These conditions are more easily verifiable than those given in the Definition \ref{solitondef2}.

\begin{proposition}\cite[Proposition 7.4]{LAURET7}\label{sufsoliton}
  If an almost-K\"{a}hler structure $(\ip,J,\omega)$ on a Lie algebra $\ggo$ satisfies any of the following conditions
\begin{eqnarray}
\label{cond1}  \left\{\begin{array}{l} \chP + \Ric^{ac} = c\Id + D \end{array}  \right.  \\
\nonumber \mbox{or}\\
\label{cond2}  \left\{\begin{array}{l} \chP= c_{1}\Id + D_{1}  \\  \Ric^{ac} = c_{2}\Id + D_{2}\end{array}  \right.
\end{eqnarray}
with $c\mbox{'s}\in \RR$ and $D\mbox{'s}\in \Der(\ggo)$, then $(\ip,J,\omega)$ defines a soliton with same $c$ and $D$ in the Condition (\ref{cond1}), and $c=c_{1}+c_{2}$ and $D=D_{1}+D_{2}$  in the Condition (\ref{cond2}).
\end{proposition}

Let $(\ngo,\omega)$ be a symplectic nilpotent Lie algebra and $\ip$ be a compatible metric with $(\ngo,\omega)$. If the anti-complexified part of the Ricci operator of $\ip$, $\Ric^{ac}$,  satisfies $\Ric^{ac}=c\Id+D$ for some $c\in \RR$ and $D \in \Der(\ngo)$, then such metric is \textit{minimal} (see \cite[Theorem 4.3]{LAURET3}). Thus, in the nilpotent case, some soliton almost-K\"{a}hler structures are given by minimal compatible metrics with the Chern-Ricci operator being a derivation.

\begin{example}
Consider the nilpotent Lie algebra $\ngo:=(\RR^8,\mu)$ with
\begin{eqnarray*}
\mu &=& \left\{\begin{array}{l}[{e_1},{e_2}]=\frac{\sqrt {14}}{14}{e_4},[{e_2},{e_5}]=\frac{\sqrt {14}}{14}{e_8},[{e_2},{e_6}]=\frac{\sqrt {14}}{14}{e_3},[{e_3},{e_7}]=\frac{\sqrt {14}}{14}{e_4},\\
{[{e_5},{e_7}]}=-\frac{\sqrt {14}}{14}{e_6},[{e_6},{e_7}]=-\frac{\sqrt {14}}{14}{e_1},[{e_7},{e_8}]=\frac{\sqrt {14}}{14}{e_3}
\end{array}\right.
\end{eqnarray*}
and the almost K\"{a}hler structure given by $(\ip,J,\omega_{cn})$ where $\ip$ is the canonical inner product of $\RR^{8}$ and
$\omega_{cn}$ is the usual symplectic form of $\RR^8$; $\omega_{cn} = e_{1}^{\ast}\wedge e_{8}^{\ast} + e_{2}^{\ast}\wedge e_{7}^{\ast} +e_{3}^{\ast}\wedge e_{6}^{\ast} +e_{4}^{\ast}\wedge e_{5}^{\ast}.$ Let $\widehat{H}=\frac{\sqrt{14}}{28}e_{7}$, which is such that $\chPform(X,Y)=\omega_{cn}(\widehat{H},[X,Y])$.
Since $e_{7}\perp^{\omega}[\ngo,\ngo]$, $(\ngo, \ip, J,\omega)$ is Chern-Ricci flat.

The anti-complexified part of the Ricci operator of $(\ngo,\ip)$ is such that
\begin{eqnarray*}
  \Ric^{ac} &=& \frac{1}{56}\diag(0,1,2,4,-4,-2,-1,0)\\
            &=& -\frac{3}{56}\Id + \frac{1}{56}\diag(3,4,5,7,-1,1,2,3)
\end{eqnarray*}
with $\diag(3,4,5,7,-1,1,2,3)$ being a derivation of $\ngo$.

It follows from Proposition \ref{sufsoliton} that $(\ip,J,\omega_{cn})$ is a
soliton K\"{a}hler structure on $\ngo$.
\end{example}

\begin{remark}
In the theory of nilsoliton metrics; minimal metrics on nilpotent Lie algebras, it is well-known that the eigenvalues of the \textit{Einstein derivation} are all positive integers (up to a positive multiple). More precisely, if $\ip$ is a nilsoliton metric on a nilpotent Lie algebra $\ngo$, with $\Ric=c\Id + D$ where $c\in\RR$ and $D \in \Der(\ngo)$, then there exist a positive constant $k$, such that all eigenvalues of $D$ lie in $k\NN$ (\cite[Theorem 4.14]{HEBER1}). The above example shows a subtle difference between nilsoliton metrics on nilpotent Lie algebras and minimal compatible metrics with symplectic Lie algebras.
\end{remark}

\subsection{Canonical compatible metrics for geometric structures on nilmanifolds}\label{canonicalmetric}

In this section we give a brief exposition of \textit{minimal compatible metrics}
for geometric structures on nilmanifolds (\cite{LAURET3}). Such approach is a way to study the problem of finding \comillas{the best metric} which is compatible with a fixed geometric $\gamma$ on a simply connected nilpotent Lie group. By using strong results from real geometric invariant theory (real GIT for short), the properties that make a minimal metric \comillas{special} are given in \cite{LAURET3}: a minimal metric is unique (up to isometry and scaling) when it exists, and it can be characterized as a \textit{soliton solution} of the \textit{invariant Ricci flow} (\cite[Theorem 4.4]{LAURET3}).

By using results given in \cite{FERNANDEZ-CULMA2}, we introduce the notion of \textit{nice basis} (Definition \ref{gamanice}) in the context of minimal metrics and give the corresponding criterion to know when a geometric structure $\gamma$ on a nilpotent Lie algebra admitting a $\gamma$-nice basis has a minimal compatible metric.

Let $(\GrpN,\gamma)$ be a class $\gamma$-nilpotent Lie group: $\GrpN$ is a simply connected nilpotent Lie group and $\gamma$ is an invariant geometric structure on $\GrpN$ (see \cite[Definition 2.1]{LAURET3}). We identify $\ngo$ with $\RR^n$ and so the structure of Lie algebra on $\ngo$ is given by a element $\mu \in \lamn$; $\ngo=(\RR^n,\mu)$ and the geometric structure $\gamma$ is given by left translation of a tensor on $\RR^n$ which we denote also by $\gamma$. In the same way, any left invariant compatible metric with $(\GrpN,\gamma)$ is defined by a inner product $\ipp$ on $\RR^n$.

By definition, there is no loss of generality in assuming that the canonical inner product of $\RR^{n}$, $\ip$, also defines a compatible metric with $(\GrpN,\gamma)$. Since the reductive group
\begin{eqnarray*}
\GrpG_{\gamma}=\{g \in \Glr : g\cdot \gamma = \gamma\}
\end{eqnarray*}
is self adjoint with respect to any compatible metric (it also follows easily from definition), then $\GrpG_{\gamma}$ is compatible with the usual Cartan decomposition of $\Glr$; it is to say: $\GrpG_{\gamma}=\GrpK_{\gamma} \exp(\ag_{\gamma}) \GrpK_{\gamma}$ with $\GrpK_{\gamma}$ a subgroup of the Orthogonal group $\Or$ and $\ag_{\gamma}$ a subalgebra of the diagonal matrices $\ag$.

\begin{example}\label{ejemploLibre}
Consider the free $2$-step nilpotent Lie algebra of $\rank 3$; $\ngo_{18}:=[e_1,e_2] =e_4, [e_1,e_3] = e_5, [e_2,e_3] = e_6$ and the
symplectic structure $\omega_{2}(s)$ on $\ngo_{18}$ given by
\begin{eqnarray*}
\omega_{2}(t)= e_1^{\ast}\wedge e_5^{\ast} + t e_1^{\ast}\wedge e_6^{\ast} -t e_2^{\ast}\wedge e_5^{\ast} + e_2^{\ast}\wedge e_6^{\ast} - 2te_3^{\ast}\wedge e_4^{\ast} \mbox{ \cite[Theorem 5]{KHAKIMDJANOV1}}.
\end{eqnarray*}
In general, it is well-known that given any symplectic form $\omega$, there exists a suitable change of basis such that $\omega$ is given by
the \comillas{canonical symplectic form}. In this case, we can try to do a change of basis of the form
$$
g=\diag\left(m_{{1,1}},m_{{2,2}},m_{{3,3}},m_{{4,4}},\left(\begin{array}{cc}m_{{5,5}}&m_{{5,6}}\\m_{{6,5}}&m_{{6,6}} \end{array}\right)\right)
$$
and to solve $g\cdot\omega_{2}(t)=\omega_{2}(t)(\cdot,\cdot)=\omega_{cn}$ for $ \{m_{{1,1}},\ldots,m_{{6,6}} \}$. Here, $\omega_{cn}=e_{1}^{\ast}\wedge e_{6}^{\ast} + e_{2}^{\ast}\wedge e_{5}^{\ast} +e_{3}^{\ast}\wedge e_{4}^{\ast}$.

The solution to this equation is given by
$$\begin{array}{l}
 \left\{ m_{{1,1}}=m_{{1,1}},m_{{2,2}}=m_{{2,2}},m_{{3,3}}=m_{{3,3}},m
_{{4,4}}=-\frac{1}{2}\,{\frac {1}{tm_{{3,3}}}}, \right.\\
\left.m_{{5,5}}=-{\frac {t}{ \left( {
t}^{2}+1 \right) m_{{2,2}}}},m_{{5,6}}={\frac {1}{ \left( {t}^{2}+1
 \right) m_{{1,1}}}},m_{{6,5}}={\frac {1}{ \left( {t}^{2}+1 \right) m_
{{2,2}}}},m_{{6,6}}={\frac {t}{ \left( {t}^{2}+1 \right) m_{{1,1}}}}
 \right\},
\end{array}$$
hence, we can take the particular solution defined by $ m_{{1,1}}=m_{{2,2}}=m_{{3,3}}=1$, which defines symplectomorphism from $(\ngo_{18},\omega_{2}(t))$ to $(\RR^6,\mu_{t},\omega_{cn})$ with
\begin{eqnarray*}
 \mu_{t} &=& \left\{[{e_1},{e_2}]=-2\,t\,{e_4},[{e_1},{e_3}]=-
t\,{e_5}+{e_6},[{e_2},{e_3}]={e_5}+t\,{e_6}\right..
\end{eqnarray*}

With respect to $\GrpG_{\omega_{cn}}$, we have the usual presentation of the \textit{symplectic group}, $\Spe(3,\RR)$:
\begin{eqnarray*}
 \Spe(3,\RR) &=&\left\{g \in \mathrm{GL}_{6}(\RR) / g^{\trans}Jg=J\right\}
\end{eqnarray*}
where $Je_{1}=e_{6}$, $Je_{2}=e_{5}$, $Je_{3}=e_4$ and $J^{2}=-\Id$, \comillas{the} maximal compact subgroup of $\Spe(3,\RR)$ is given by
the unitary group $\U(n)$ and a Cartan decomposition given by $\Spe(3,\RR) = \U(n)\exp(\ag_{\omega_{cn}})\U(n)$ with
\begin{eqnarray*}
  \ag_{\omega_{cn}}&=& \left\{ \diag(-x_{1},-x_{2},-x_{3},x_{3},x_{2},x_{1}) : x_{i} \in \RR \right\}.
\end{eqnarray*}
\end{example}

\begin{definition}\cite[Definition 2.2]{LAURET3}\label{defRicci}
Let $\ipp$ be a \textit{compatible metric} with the class-$\gamma$ nilpotent Lie group $(\GrpN,\gamma)$. Consider the orthogonal projection $\Ric^{\gamma}_{\ipp}$ of the Ricci operator $\Ric_{\ipp}$ on $\ggo_{\gamma} = \Lie(\mathrm{G}_{\gamma})$ with respect to the inner product $\ippd$ of $\glr$ induced by $\ipp$; i.e. for any $A$, $B$ in $\glr$, ${( \hspace{-0.5mm} (A, B) \hspace{-0.5mm} )}=\tr(AB^{\trans})$ where $B^{\trans}$ denote the transpose of $B$ with respect to $\ipp$. $\Ric^{\gamma}_{\ipp}$ is said to be \textit{invariant Ricci operator}, and the corresponding \textit{invariant Ricci tensor} is given by $\ricci^{\gamma}=\ippa{\Ric^{\gamma}\cdot}{\cdot}$.
\end{definition}

\begin{example}
  In the symplectic case, it is easy to see that the invariant Ricci operator coincides with the anti-complexified Ricci tensor, i.e. if $\ipp$ is a compatible metric with $(\ngo, \omega)$, then
\begin{eqnarray*}
\Ric^{\omega}_{\ipp}&=&\Ric^{ac}_{\ipp}=\frac{1}{2}\left(\Ric_{\ipp} + J_{\ipp}\Ric_{\ipp}J_{\ipp}\right)
\end{eqnarray*}
where  $J_{\ipp}$ is the linear transformation such that $\omega(\cdot,\cdot) = (J_{\ipp}\cdot,\cdot)$ .
\end{example}

\begin{definition}[\textbf{Minimal compatible metric}]\cite[Definition 2.3]{LAURET3}
A left invariant metric $\ip$ compatible with a class-$\gamma$ nilpotent Lie group $(\GrpN_{\mu},\gamma)$ is called \textit{minimal} if
$$
\tr(\Ric^{\gamma}_{\ip})^2=\min\left\{\tr(\Ric^{\gamma}_{\ipp})^2:\mbox{\begin{tabular}{l} $\ipp$ is a compatible metric with  $(\GrpN_{\mu},\gamma)$ \\
and $\scalar(\ipp)=\scalar(\ip)$
\end{tabular} } \right\}
$$
\end{definition}

Now we study the natural action of $\Glr$ (and $\GrpG_{\gamma}$) on $V:=\Lambda^2(\RR^{n})^*\otimes\RR^{n}$ given by the change of basis:
$$
g\cdot \mu(X,Y)=g \mu(g^{-1} X, g^{-1} Y), \, X,Y \in \RR^n, \, g\in\Glr, \, \mu \in V.
$$
The corresponding representation of $\g$ on $V$ is given by
$$
A\cdot\mu(X,Y)=A\mu(X,Y)-\mu(AX,Y)-\mu(X,AY), \, A\in\g \, \mu\in V.
$$

Consider the usual inner product $\ip$ on $V$ which is defined by the canonical inner product of $\RR^n$ as follows:
$$
\ipa{\mu}{\lambda}=\sum_{ijk}\ipa{\mu(e_i,e_j)}{e_k}\ipa{\lambda(e_i,e_j)}{e_k}, \, \mu, \lambda \in V.
$$
and let $\ipd$ the canonical inner product of $\g$ induced by the canonical inner product of $\RR^{n}$ (as in the Definition \ref{defRicci}).

We are now in a position to define the moment map of the above-mentioned action. This map is implicitly defined by
\begin{equation}\label{momentmap}
\begin{array}{rcl}
  \mm_{\g}:V &\longrightarrow &\g  \\
\ipda{\mm_{\g}(\mu)}{A} &=& \ipa{A\cdot \mu}{\mu},
\end{array}
\end{equation}
for all $A\in\g$ and $\mu\in V$.

Let $\Proj_{\ggo_{\gamma}}$ be denote the orthogonal projection of $\g$ on $\ggo_{\gamma}$ with respect to the inner product $\ipd$, then it is easy to see that the moment map for the action of $\GrpG_{\gamma}$ on $V$, $\mm_{\ggo_{\gamma}}$, is  $\Proj_{\ggo_{\gamma}} \circ\mm_{\g}$. The relationship between minimal metrics and the moment map is given by the following result

\begin{proposition}\cite[Propositon 4.2]{LAURET3}\label{riccimom}
Let $(\mathrm{N}_{\mu},\gamma)$ be a class-$\gamma$ nilpotent Lie group. Then
\begin{eqnarray}
4\Ric_{g\cdot \ip} &=&\mm_{\g}(g^{-1}\cdot\mu), \, \forall g\in \Glr\\
4\Ric ^{\gamma}_{h\cdot\ip} &=&\mm_{\mathfrak{g}_{\gamma}}(h^{-1}\cdot\mu), \, \forall h \in \GrpG_{\gamma}
\end{eqnarray}
where $\Ric_{g\cdot \ip}$ is the Ricci operator of the Riemannian manifold $(\mathrm{N}_{\mu},g\cdot\ip)$ with respect to the orthonormal basis $\{g\cdot e_1, \ldots , g \cdot e_n\}$ and $\Ric^{\gamma}_{h\cdot\ip}$ is the invariant Ricci operator of $(\mathrm{N}_{\mu},\gamma, h\cdot\ip)$ with respect to the orthonormal basis $\{h\cdot e_1, \ldots , h\cdot e_n\}$.
\end{proposition}

Hence, the problem of finding a minimal compatible metric with $(\mathrm{N}_{\mu},\gamma)$ is equivalent to find a minimum value of $||\mm_{\ggo_{\gamma}}||^2$ along the $\GrpG_{\gamma}$-orbit of $\mu$ (we recall that any compatible metric is of the form $h\cdot\ip$ with $h \in \GrpG_{\gamma}$). The above is exactly to know if the orbit $\GrpG_{\gamma}\cdot \mu$ is \textit{distinguished} for the action of $\GrpG_{\gamma}$ on $V$ ($\GrpG_{\gamma}\cdot \mu$ contains a critical point of $||\mm_{\ggo_{\gamma}}||^2$).

\begin{theorem}\cite[Theorem 4.3 and 4.4]{LAURET3}\label{minimalmetric}
Let $(\mathrm{N}_{\mu},\gamma)$ be a class-$\gamma$ nilpotent Lie group. $(\mathrm{N}_{\mu},\gamma)$ admits a minimal compatible metric if and only if the $\GrpG_{\gamma}$-orbit of $\mu$ is distinguished for the natural action of $\GrpG_{\gamma}$ on $V$. Moreover, there is at most one minimal compatible metric on $(\mathrm{N},\gamma)$ up to isometry (and scaling).
\end{theorem}

\begin{remark}
The last part of the above theorem \label{minimalmetric} follows from strong results on critical points of the norm-square of a moment map. In \cite[Proposition 4.3 and 4.4]{LAURET3} is used a result of Alina Marian (\cite[Theorem 1]{MARIAN1}) to prove such part.  However, there is an error in the proof of Marian's Theorem 1. A correct proof can be found in \cite[Theorem 5.1]{JABLONSKI4} and \cite[Corollary 6.12]{HEINZNER3}.
\end{remark}

We are now in a position to introduce nice basis notion.

\begin{definition}[\textbf{$\gamma$-nice basis}]\label{gamanice}
We say that the canonical basis of $\RR^{n}$, $\{e_1,\ldots,e_n\}$, is a $\gamma$-nice basis of $(\RR^{n},\mu,\gamma)$ if for any metric of the form $a\cdot \ip$ with $a\in \exp(\ag_{\gamma})$ one has that  $\Ric^{\gamma}_{a\cdot\ip} \in \ag_{\gamma}$, where $\Ric^{\gamma}_{a\cdot\ip}$ is represented with respect to the orthonormal basis $\{a\cdot e_1,\ldots,  a\cdot e_n\}$ of $(\RR^n,\mu,a\cdot\ip)$.
\end{definition}

\begin{remark}
By Propositions \ref{riccimom} and \cite[Proposition 4.8]{FERNANDEZ-CULMA2}, the above definition is equivalent to saying that $\mu$ is a nice-element
for the natural action of $\GrpG_{\gamma}$ on $V$ (\cite[Defintion 3.3]{FERNANDEZ-CULMA2}).
\end{remark}

\begin{remark}
In general, it is difficult to know when a pair $(\GrpN,\gamma)$ admits a $\gamma$-nice basis, even if $\gamma=0$ (nilsoliton case). In \cite[Section 4]{FERNANDEZ-CULMA2} we study this problem in the general case of real reductive representations and some results obtained will be very useful in the study of minimal metrics.
\end{remark}

\begin{notation}
Let $(\RR^{n},\mu,\gamma)$ be such that $\{e_1,\ldots,e_n\}$ is a $\gamma$-nice basis. Let us denote by $\mathfrak{R}_{\gamma}(\mu)$ the ordered set of weights related with $\mu$ to the action of $\mathrm{G}_{\gamma}$ on $V$ (see \cite[Notation 2.5]{FERNANDEZ-CULMA2}); i.e. if $\{C_{i,j}^{k}\}$ are the structural constants of $(\RR^{n},\mu)$ in the basis $\{e_1,\ldots,e_n\}$ then
$$
\mathfrak{R}_{\gamma}(\mu):=\left\{\Proj_{\ggo_{\gamma}} (E_{k,k}-E_{i,i}-E_{j,j}) : C_{i,j}^{k}\neq 0\right\}
$$
where $\{E_{i,j}\}$ is the canonical basis of $\g$.

We denote by $\beta_{\mu}^{\gamma}$ the \textit{minimal convex combination} of the convex hull of $\mathfrak{R}_{\gamma}(\mu)$; i.e. $\beta_{\mu}^{\gamma}$ is the unique vector closest to the origin in the mentioned hull.

The Gram matrix of $(\mathfrak{R}_{\gamma}(\mu),\ipd)$ will be denoted by $\gramU_{\mu}^{\gamma}$; i.e. if $\mathfrak{R}_{\gamma}(\mu)_{p}$ is $p$-th element of $\mathfrak{R}_{\gamma}(\mu)$, then
$$
\gramU_{\mu}^{\gamma}(p,q)=\ipda{\mathfrak{R}_{\gamma}(\mu)_{p}}{\mathfrak{R}_{\gamma}(\mu)_{q}}
$$
\end{notation}

By using the above notation, it follows from \cite[Theorem 3.14]{FERNANDEZ-CULMA2} an aim in this note.

\begin{theorem}\label{nicegeneralizado}
Let $(\RR^{n},\mu,\gamma)$ be such that $\{e_1,\ldots,e_n\}$ is a $\gamma$-nice basis.
$(\mathrm{N}_{\mu},\gamma)$ admits a compatible minimal metric if and only if the equation
\begin{eqnarray}\label{eqgram}
\gramU^{\gamma}_{\mu}[x_i]=\lambda[1]
\end{eqnarray}
has a positive solution $[x_i]$ for some $\lambda \in \RR$. Moreover, in such case, there exists $a\in \exp(\ag_{\gamma})$ such that $a\cdot\ip$ defines a minimal compatible metric with $(\mathrm{N}_{\mu},\gamma)$.
\end{theorem}

\begin{remark}
The proof of Theorem \cite[Theorem 3.14]{FERNANDEZ-CULMA2} gives more, namely if $(\RR^{n},\mu,\gamma)$ admits a minimal compatible metric, then one can find a such metric by solving the equation
\begin{eqnarray}\label{equminimal}
\mm_{\ggo_{\gamma}}(a\cdot\mu)=\beta_{\mu}^{\gamma}
\end{eqnarray}
for $a\in \exp(\ag_{\gamma})$. Since,
$$
\beta_{\mu}^{\gamma}=\frac{1}{\sum x_p}\left( \sum x_p \mathfrak{R}_{\gamma}(\mu)_{p} \right)
$$
where $[x_i]$ is any positive solution to the Equation (\ref{eqgram}), in practice it is sometimes easy to solve the Equation (\ref{equminimal}).
\end{remark}

\section{Soliton almost K\"{a}hler structures}

In this section, we want to give those soliton almost K\"{a}hler structures on two- and three-step nilpotent Lie algebras of dimension $6$ that are obtained by minimal compatible metrics with symplectic structures. By following the classification given in \cite{KHAKIMDJANOV1} for $6$-dimensional symplectic nilpotent Lie algebras, a simple inspection of such classification list reveals that many pairs $(\ngo,\omega)$ are written in a $\omega_{cn}$-nice basis or that by using a suitable change of basis, these can be written in a nice basis. Here $\omega_{cn}$ is the canonical symplectic form of $\RR^{6}$; $\omega_{cn}=e_1^{\ast} \wedge e_{6}^{\ast} + e_2^{\ast} \wedge e_{5}^{\ast} + e_3^{\ast} \wedge e_{4}^{\ast} $.

We denote by $\mm_{\spg}$ the moment map corresponding to the action of the symplectic group $\GrpG_{\omega_{cn}}=\Spe(3,\RR)$ on $V$, and we have
$$
\ag_{\omega_{cn}}=\{\diag(-x_1,-x_2,-x_3,x_3,x_2,x_1) : x_i \in \RR \}\}
$$

Theorem \ref{nicegeneralizado} has been applied to each mentioned algebra individually. We will consider in detail only some examples, which we think to be representative; remaining cases are established in an entirely analogous way.

\subsection{Symplectic three-step nilpotent Lie algebras}

In this part, first we give a complete classification of minimal compatible metrics on symplectic three-step nilpotent Lie algebras. Then, we compute
the respective Chern-Ricci operator, which happens to be a derivation in this case. From Proposition \ref{sufsoliton}, we then have the proof for the Theorem \ref{th2}.

\begin{example}

We consider the nilpotent Lie algebra $\ngo_{11}$ given by $[e_1,e_2] = e_4, [e_1,e_4] = e_5,[e_2,e_3] = e_6, [e_2,e_4] = e_6$, which carries two curves of non-equivalent symplectic structures, namely $\omega_{1}(\lambda)=e^{\ast}_1 \wedge e^{\ast}_6 +e^{\ast}_2 \wedge e^{\ast}_5 + \lambda e^{\ast}_2 \wedge e^{\ast}_6 - e^{\ast}_3 \wedge e^{\ast}_4,$ with $\lambda \in \RR$ and $\omega_{2}(\lambda)=-\omega_{1}(\lambda)$ (By \cite[Theorem 5]{KHAKIMDJANOV1}).  Let us see the case of $\omega_{1}(\lambda)$, similar considerations apply to the other case.

In our approach, we need that the canonical inner product defines a compatible metric, which is similar to give a basis of $\ngo_{11}$ where the symplectic structure is defined by $\omega_{cn}$. To do this, we can try with a change of basis of the form
$$
g^{-1}=\diag\left( \left(\begin{array}{cc}m_{{1,1}}&m_{{1,2}}\\0 &m_{{2,2}} \end{array}\right), \left(\begin{array}{cc}m_{{3,3}}&m_{{3,4}}\\m_{{4,3}}&m_{{4,4}} \end{array}\right), \left(\begin{array}{cc}m_{{5,5}}&m_{{5,6}}\\0&m_{{6,6}} \end{array}\right)\right).
$$
Since we also need a $\omega_{cn}$-nice basis,  we can also try to get that $\mm_{\mathfrak{sp}}(\exp(X) \cdot g\cdot\ngo_{11}) \in \ag_{\omega_{cn}}$ for any
$X \in \ag_{\omega_{cn}}$.
If $\lambda \neq 0$ then by solving such system of equations, we have, for instance, a solution given by
$$\begin{array}{l}
\left\{ m_{{1,1}}=-\frac{1}{2}\,{\frac {\lambda}{m_{{5,6}}}},m_{{1,2}}=-\frac{1}{2}\,m_{{2,2}}\lambda,
m_{{2,2}}=m_{{2,2}},m_{{3,3}}=-{m_{{4,4}}}^{-1},m_{{3,4}}=0, \right.\\
\left.m_{{4,3}}=\frac{1}{2}\,{m_{{4,4}}}^{-1},m_{{4,4}}=m_{{4,4}},m_{{5,5}}={
m_{{2,2}}}^{-1},m_{{5,6}}=m_{{5,6}},m_{{6,6}}=-2\,{\frac {m_{{5,6}}}{
\lambda}} \right\}.
\end{array}
$$
If we let $m_{{2,2}}=1$, $m_{{4,4}}=1$, and $m_{{5,6}}=1$ then
$$
g=\left( \left(\begin{array}{cc} -\frac{2}{\lambda}&-1\\0& 1\end{array}\right), \left(\begin{array}{cc} -1 & 0 \\ \frac{1}{2}&1 \end{array}\right),\left(\begin{array}{cc}1&\frac{\lambda}{2}\\0&-\frac{\lambda}{2} \end{array}\right), \right)
$$
defines a symplectomorphism from $(\ngo_{11},\omega_{1}(\lambda))$ to $(\RR^{6},\mu_{\lambda},\omega_{cn})$ where
$$
\mu_{\lambda} := \left\{
\begin{array}{l}
 [{e_1},{e_2}]=-\frac{1}{2}\,\lambda\,{e_4},
 [{e_1},{e_3}]=-\frac{1}{4}\,\lambda\,{e_5},
 [{e_1},{e_4}]=-\frac{1}{2}\,\lambda\,{e_5},\\
{[{e_2},{e_3}]}=-\frac{1}{2}\,\lambda\,{e_5} + \frac{1}{4}\,\lambda\,{e_6},
 [{e_2},{e_4}]=-\frac{1}{2}\,\lambda\,{e_6}\\
\end{array}
\right.
$$
and $(\RR^{6},\mu_{\lambda},\omega_{cn})$ is written in a $\omega_{cn}$-nice basis.
For all $\lambda \in \RR \setminus \{0\}$, the Gram matrix is given by
$$
\gramU^{\omega_{cn}}_{\mu_{\lambda}} = \frac{1}{2}\left[ \begin {array}{ccc} 3&1&1\\ \noalign{\medskip}1&5&3
\\ \noalign{\medskip}1&3&3\end {array} \right]
$$
and since the general solution to $\gramU^{\omega_{cn}}_{\mu_{\lambda}}X=[1]$ is $X =\frac{1}{2}[1,0,1]^{\trans}$, for any $\lambda \neq 0$, $(\ngo_{11},\omega_{1}(\lambda))$ does not admit a minimal metric.

When $\lambda=0$, on the contrary, $(\ngo_{11},\omega_{1}(\lambda=0))$ admit a minimal metric. In fact, like above, we consider $g=\diag\left(1,1,\left(\begin{array}{cc}1&\frac{1}{2}\\0&-1 \end{array}\right),1,1\right)$. $g$ defines a symplectomorphism from $(\ngo_{11},\omega_{1}(0))$ to $(\RR^{6},\mu,\omega_{cn})$ with
$$
\mu_{} := \left\{
\begin{array}{l}
[{e_1},{e_2}]=\frac{1}{2}\,{e_3}-{e_4},
[{e_1},{e_4}]=-{e_5},
[{e_2},{e_3}]={e_6},
[{e_2},{e_4}]=-\frac{1}{2}\,{e_6}
\end{array}
\right.
$$
where $(\RR^{6},\mu,\omega_{cn})$ is written in a $\omega_{cn}$-nice basis.
The Gram matrix is given by
$$
\gramU^{\omega_{cn}}_{\mu_{}} = \frac{1}{2}\left[ \begin {array}{cc} 3&1\\ \noalign{\medskip}1&3\end {array}
 \right]
 $$
and the general solution to $\gramU^{\omega_{cn}}_{\mu_{\lambda}}X=[1]$ is $X=\frac{1}{2}[1,1]^{\trans}$. Since $X^{\prime}=X$ is a positive solution, $(\ngo_{11},\omega_{1}(0))$ admits a minimal metric. To find such metric, we solve the problem
$$
4\Ric^{ac}(\exp(Y)\cdot \mu) = \mm_{\spg}(\exp(Y)\cdot \mu)= \mcc(\mathfrak{R}_{\omega_{cn}}(\mu))
$$
with $Y \in \ag_{\omega_{cn}}$. Let $Y=\diag(\ln(2)+\frac{1}{4}\ln(3),0,-\frac{1}{4}\ln(3)+\frac{1}{2}\ln(2),\frac{1}{4}\ln(3)-\frac{1}{2}\ln(2),0,-\ln( 2 ) -\frac{1}{4}\, \ln(3))$, the change of basis given by $\exp(Y)$ defines
$$
\widetilde{\mu} := \left\{
\begin{array}{l}
[{e_1},{e_2}]=\frac{\sqrt {6}}{12}{e_3}- \frac{\sqrt {2}}{4}{e_4},
[{e_1},{e_4}]=-\frac{\sqrt {6}}{6}{e_5},
[{e_2},{e_3}]=\frac{\sqrt {2}}{4}{e_6},
[{e_2},{e_4}]=-\frac{\sqrt {6}}{12}{e_6}
\end{array}
\right.
$$
Since
$$
\mm_{\glg}(\widetilde{\mu}) =\frac{1}{6} \diag(-4,-4,-1,-1,2,2),
$$
it follows that
$$
\begin{array}{rcl}
\mm_{\spg}(\widetilde{\mu}) & = & \frac{1}{2}(\mm_{\glg}(\widetilde{\mu})+J.\mm_{\glg}(\widetilde{\mu}).J)\\
                            &&\\
                            & = & \frac{1}{2}\diag(-1,-1,0,0,1,1)\\
                            &&\\
                            & = & -\Id + \underbrace{\frac{1}{2}\diag(1,1,2,2,3,3)}_{\mbox{Derivation}}
\end{array}
$$
and thus, the canonical inner product of $\RR^6$ defines a minimal metric on $(\RR^6,\widetilde{\mu},\omega_{cn})$ .

A straightforward computation shows that $(\RR^{6},\mu_{\lambda},\ip,J,\omega_{cn})$ and $(\RR^{6},\widetilde{\mu},\ip,J,\omega_{cn})$ are Chern-Ricci flat, where $\ip$ is the canonical inner product of $\RR^6$.
\end{example}

\begin{example}\label{ejemplofeo}
We consider the nilpotent Lie algebra $\ngo_{13}$ given by $[e_1,e_2] = e_4, [e_1,e_3] = e_5, [e_1,e_4] = e_6, [e_2,e_3] = e_6$ and the curve of non-equivalent symplectic structures $\omega_{2}(\lambda)$ with $\lambda \neq 0$:
$$
\omega_{2}(\lambda) = e^{\ast}_1 \wedge e^{\ast}_6 +\lambda e^{\ast}_2 \wedge e^{\ast}_4 +e^{\ast}_2 \wedge e^{\ast}_5 + e^{\ast}_3 \wedge e^{\ast}_5.
$$
The change of basis given by $g=\left(1,\left(\begin{array}{cc}\lambda-\frac{1}{2}&-\frac{1}{2}\\1&1 \end{array}\right),\left(\begin{array}{cc}\frac{1}{2}&1\\1& 0\end{array}\right),1\right)$ defines a symplectomorphism from $(\ngo_{13},\omega_{2}(\lambda))$ to $(\RR^{6},\mu_{\lambda},\omega_{cn})$ with
$$
\mu_{\lambda} := \left\{
\begin{array}{l}
[{e_1},{e_2}]=-{\frac {1}{2\lambda}}{e_4} + {\frac {1}{\lambda}}{e_5},
[{e_1},{e_3}]={\frac { \left( 4\,\lambda-1 \right) }{4\lambda}}{e_4} + {\frac {1}{2\lambda}}{e_5},
[{e_1},{e_5}]={e_6},\\
{[{e_2},{e_3}]}={\frac {1}{\lambda}}{e_6}.
\end{array} \right.
$$
\end{example}
It is a simple matter to see that $(\RR^{6},\mu_{\lambda},\omega_{cn})$ is written in a $\omega_{cn}$-nice basis and that if $\lambda \neq \frac{1}{4}$ then the Gram matrix is given by
$$
\gramU^{\omega_{cn}}_{\mu_{\lambda}} = \frac{1}{2}\left[ \begin {array}{cccc} 3&3&3&1\\ \noalign{\medskip}3&5&1&0
\\ \noalign{\medskip}3&1&5&2\\ \noalign{\medskip}1&0&2&5\end {array}
 \right].
$$
The general solution to $\gramU^{\omega_{cn}}_{\mu_{\lambda}}X=[1]$ is $X=\frac{1}{25}[10-50\,t , 4+25\,t , 25\,t , 8 ]^{\trans}$; since $X^{\prime}=\frac{2}{25}[3,3,1,4]^{\trans}$ is a positive solution, $(\ngo_{13},\omega_{2}(\lambda))$ admits a minimal metric.

Although in this case, it is difficult to give an explicit formula for such curve of metrics, we can say that if such metrics have scalar curvature equal to $-\frac{1}{4}$ then $4\Ric^{ac} = \mm_{\mathfrak{sp}} = -\frac{25}{22}\Id + \frac{5}{11}\diag(1,2,2,3,3,4)$. Furthermore, they are given in the family of symplectic nilpotent Lie algebras $(\RR^6,\mu_{t},\omega_{cn})$ with
$$
\mu_{t}:=\left\{
\begin{array}{l}
[{e_1},{e_2}]=-t{e_4} \pm \frac{1}{22}\,\sqrt {99-1452\,{t}^{2}}{e_5},
[{e_1},{e_3}]= \pm \frac{1}{22}\,\sqrt {55-1452\,{t}^{2}}{e_4}+t{e_5},\\
{[{e_1},{e_5}]}= \frac{\sqrt {22}}{11}{e_6},
[{e_2},{e_3}]=2\,t{e_6}
\end{array}  \right.
$$
The Chern-Ricci operator of $(\RR^6,\mu_{t},\ip,J,\omega_{cn})$ is given by
\begin{eqnarray*}
\chP&=&\frac{t\sqrt{22}}{22}(E_{4,1}-E_{6,3})\pm{\frac {\sqrt {22}}{484}}\,\sqrt {99-1452\,{t}^{2}}(E_{5,1}-E_{6,2})
\end{eqnarray*}
which is easily seen to be a derivation of $(\RR^6,\mu_{t})$. From Proposition \ref{sufsoliton}, we have $(\ip,J,\omega_{cn})$ defines a soliton K\"{a}hler structure on $(\RR^6,\mu_{t})$.

If $\lambda = \frac{1}{4}$, one can proceed as above and showing that $(\ngo_{13},\omega_{2}(\lambda = \frac{1}{4}))$ admits a soliton K\"{a}hler structure.

\begin{example}
For a final example, Consider the nilpotent Lie algebra $\ngo_{12}$ given by $[e_1,e_2] = e_4, [e_1,e_4] = e_5, [e_1,e_3] = e_6, [e_2,e_3]=-e_5, [e_2,e_4] = e_6$ and the curve of non-equivalent symplectic structures $\omega_{1}(\lambda)=\lambda e^{\ast}_1 \wedge e^{\ast}_5 + e^{\ast}_2 \wedge e^{\ast}_6 + (\lambda +1)e^{\ast}_3 \wedge e^{\ast}_4$ (with $\lambda \in \RR \setminus \{-1,0\}$). Consider the change of basis given by
$$
g=\diag\left(1, 1,1, \lambda+1, \left(\begin{array}{cc}0&1\\ \lambda &0 \end{array}\right) \right)
$$
which defines a defines a symplectomorphism from $(\ngo_{12},\omega_{1}(\lambda))$ to $(\RR^{6},\mu_{\lambda},\omega_{cn})$ where
$$
\mu_{\lambda} := \left\{
\begin{array}{l}
 [{e_1},{e_2}]= \left( \lambda+1 \right) {e_4},
          [{e_1},{e_3}]={e_5},
          [{e_1},{e_4}]={\frac {\lambda}{\lambda+1}}{e_6},\\
          {[{e_2},{e_3}]}=-\lambda\,{e_6},
          [{e_2},{e_4}]={\frac {1}{\lambda+1}}{e_5}
\end{array} \right.
$$
As above, $(\RR^{6},\mu_{\lambda},\omega_{cn})$ is written in a $\omega_{cn}$-nice basis, and by Theorem \ref{nicegeneralizado}, one can show that $(\ngo_{12},\omega_{1}(\lambda))$ admits a minimal compatible metric. By solving $\mm_{\spg}(\exp(Y)\cdot \mu_{\lambda})= \mcc(\mathfrak{R}_{\omega_{cn}}(\mu_{\lambda}))$, we have a symplectomorphism given by $\exp(Y)$ from  $(\RR^{6},\mu_{\lambda},\omega_{cn})$ to  $(\RR^{6},\widetilde{\mu}_{\lambda},\omega_{cn})$ with
$$
\widetilde{\mu_{\lambda}} := \left\{
\begin{array}{l}
[{e_1},{e_2}]=\frac{\sqrt{2}}{4}{\frac { \left( \lambda+1 \right) \sqrt {{\lambda}^{2}+\lambda+1}}{{\lambda}^{2}+\lambda+1}}{e_4},
[{e_1},{e_3}]=\frac{\sqrt{2}}{4}{\frac {\sqrt {{\lambda}^{2}+\lambda+1}}{{\lambda}^{2}+\lambda+1}}{e_5},\\
{[{e_1},{e_4}]}=\frac{\sqrt {2}}{4} \sign \left( {\frac {\lambda}{\lambda+1}} \right) {e_6},
{[{e_2},{e_3}]}=-\frac{\sqrt{2}}{4}{\frac {\lambda\,\sqrt {{\lambda}^{2}+\lambda+1}}{{\lambda}^{2}+\lambda+1}}{e_6},\\
{[{e_2},{e_4}]}=\frac{\sqrt {2}}{4}{\sign \left( \lambda+1 \right) }{e_5}
\end{array} \right.
$$
and which is such that the canonical inner product defines a minimal compatible metric with $(\RR^{6},\widetilde{\mu}_{\lambda},\omega_{cn})$.

This example is interesting in the following sense. Let $\chP_{\lambda}$ the Chern-Ricci operator of $(\RR^{6},\widetilde{\mu}_{\lambda},\ip,J,\omega_{cn})$. We have
$$
\chP_{\lambda}=\frac{1 }{16}\,{\frac {\left( 1+ \sign \left( \lambda \right)  \right)  \left| \lambda+1 \right|\sqrt {{\lambda}^{2}+\lambda+1}}{{\lambda}^{2}+\lambda+1}} (E_{5,1}-E_{6,2}).
$$
It is easy to see that $P_{\lambda}$ is a Derivation of $(\RR^{6},\widetilde{\mu}_{\lambda})$ and moreover, $(\RR^{6},\widetilde{\mu}_{\lambda},\omega_{cn})$ is Chern-Ricci flat if and only if $\lambda$ is a negative number (with $\lambda \neq -1$).
\end{example}

\begin{theorem}
The classification of minimal metrics on $6$-dimensional symplectic three-step nilpotent Lie algebras is given in the Table \ref{minmetricas3pasos}.
In each case, such metric defines a soliton almost K\"{a}hler structure, where the Chern-Ricci operator is a derivation of the respective
nilpotent Lie algebra (see Table \ref{minmetricas3pasosCHP}).
\end{theorem}

In the Table \ref{minmetricas3pasos}, each Lie algebra defines a symplectic three-step Lie algebra given by $(\RR^6,\widetilde{\mu},\omega_{cn})$
and it is such that the canonical inner product on $\RR^6$ defines a minimal metric of scalar curvature equal to $-\frac{1}{4}$. In the column $||\beta||^2$ we give the norm squared of the \textit{stratum} associated to the minimal metric and in Derivation column, we give the derivation of $(\RR^6,\widetilde{\mu})$ such that
$$
\mm_{\mathfrak{sp}_{6}(\RR)}(\widetilde{\mu})=-||\beta||^2\Id + \mbox{Derivation}.
$$
In the last column, we give the dimension of automorphism group of the symplectic three-step Lie algebra $(\RR^6,\widetilde{\mu},\omega_{cn})$.
A line means that such symplectic nilpotent Lie algebra does not admit a minimal metric.
\begin{table}[h!]
\begin{center}
\begin{tabular}{|c|l|c|c|c|}
\hline
\multicolumn{1}{ |c| }{\centering Not. } &
\multicolumn{1}{  c|    }{\centering Critical point } &
\multicolumn{1}{  c|  } {\centering {Derivation}} &
\multicolumn{1}{  c|  }{\centering $||\beta||^2$ } &
\multicolumn{1}{  m{0.5cm }|  }{\centering $\dim$\\$\Aut$}
\tabularnewline
\hline
\hline
\multirow{2}{*}{10.1.}&
$[{e_1},{e_2}]=-\frac{\sqrt{2}}{4}{e_4},[{e_1},{e_3}]=-\frac{\sqrt{2}}{4}{e_5}, [{e_1},{e_4}]=\frac{\sqrt{2}}{4}{e_6},$ &
\multirow{2}{*}{$\frac{1}{2}\diag(1,1,2,2,3,3)$}&\multirow{2}{*}{ $1$ }& \multirow{2}{*}{5}\\
&$[{e_2},{e_4}]=-\frac{\sqrt{2}}{4}{e_5} $&&&\\
\hline
\multirow{2}{*}{10.2}&
$[{e_1},{e_2}]=-\frac{\sqrt{2}}{4}{e_4},[{e_1},{e_4}]=\frac{\sqrt{2}}{4}{e_6},[{e_2},{e_3}]=\frac{\sqrt{2}}{4}{e_6},$&
\multirow{2}{*}{$\frac{1}{2}\diag(1,1,2,2,3,3)$}&\multirow{2}{*}{$ 1 $}&\multirow{2}{*}{$ 5$}\\
 &$[{e_2},{e_4}]=-\frac{\sqrt{2}}{4}{e_5}$ & & &\\
\hline
11.1 & \multirow{2}{*}{-} & \multirow{2}{*}{-} & \multirow{2}{*}{-} & \multirow{2}{*}{$5$}\\
$\lambda \neq 0$&&&&\\
\hline
11.1&
$[{e_1},{e_2}]=-\frac{\sqrt {6}}{12}{e_3}+\frac{\sqrt {2}}{4}{e_4},[{e_1},{e_3}]=\frac{\sqrt {2}}{4}{e_5},$&
\multirow{2}{*}{$\frac{1}{2}\diag(1,1,2,2,3,3)$}&$ \multirow{2}{*}{1}$& \multirow{2}{*}{$6$}\\
$\lambda = 0$&$[{e_1},{e_4}]=-\frac{\sqrt {6}}{12}{e_5},[{e_2},{e_4}]=-\frac{\sqrt {6}}{6}{e_6}$&&&\\
\hline
11.2  & \multirow{2}{*}{ - }& \multirow{2}{*}{ - }& \multirow{2}{*}{ -} & \multirow{2}{*}{$5$}\\
$\lambda \neq 0$&&&&\\
\hline
11.2 &
$ [{e_1},{e_2}]=\frac{\sqrt {6}}{12}{e_3}+\frac{\sqrt {2}}{4}{e_4},[{e_1},{e_3}]=\frac{\sqrt {2}}{4}{e_5},$&
\multirow{2}{*}{$\frac{1}{2}\diag(1,1,2,2,3,3)$}& \multirow{2}{*}{$ 1 $}& \multirow{2}{*}{$6$}\\
$\lambda = 0$&$[{e_1},{e_4}]=\frac{\sqrt {6}}{12}{e_5},[{e_2},{e_4}]=\frac{\sqrt {6}}{6}{e_6}$&&&\\
\hline
\multirow{2}{*}{12.1}
&$[{e_1},{e_2}]=f_{1}(\lambda) \left( \lambda+1 \right) {e_4},[{e_1},{e_3}]= f_{1}(\lambda){e_5},$
&\multirow{3}{*}{$\frac{1}{2}\diag(1,1,2,2,3,3)$}&\multirow{3}{*}{$1$}&\multirow{3}{*}{$5$}\\
&$[{e_2},{e_3}]=f_{1}(\lambda)(-\lambda) {e_6},$&&&\\
&$[{e_1},{e_4}]=\frac{\sqrt {2}}{4} \sign \left( {\frac {\lambda}{\lambda+1}} \right) {e_6}, [{e_2},{e_4}]=\frac{\sqrt {2}}{4}{\sign \left( \lambda+1 \right) }{e_5}.$&&&\\
\hline
\multirow{2}{*}{13.1}&
$[{e_1},{e_2}]=f_{2}(\lambda)\left( 1-\lambda \right) {e_4},[{e_1},{e_3}]= f_{2}(\lambda){e_5},$&
\multirow{2}{*}{$\frac{1}{6}\diag(5,3,6,8,11,9)$}&\multirow{2}{*}{$\frac{7}{6}$}&\multirow{2}{*}{$7$}\\
&$[{e_2},{e_3}]= f_{2}(\lambda)(\lambda){e_6}, [{e_2},{e_4}]=\frac{\sqrt {6}}{6}{\sign} \left( \lambda-1 \right) {e_5}.$&&&\\
\hline
13.2 &
\multirow{2}{*}{see Example \ref{ejemplofeo}} &
\multirow{2}{*}{$\frac{5}{11}\diag(1,2,2,3,3,4)$ } & \multirow{2}{*}{$\frac{25}{22}$}&\multirow{2}{*}{ $6$}\\
$\lambda \neq \frac{1}{4}$&&&&\\
\hline
13.2 &
$[{e_1},{e_2}]=-{\frac {\sqrt {165}}{66}}{e_4}+\frac{\sqrt {11}}{11}{e_5},[{e_1},{e_3}]={\frac {\sqrt {165}}{66}}{e_5},$&
\multirow{2}{*}{$\frac{5}{11}\diag(1,2,2,3,3,4)$} & \multirow{2}{*}{$\frac{25}{22}$} & \multirow{2}{*}{$ 6$}\\
$\lambda  =   \frac{1}{4}$&$[{e_1},{e_5}]=\frac{\sqrt {22}}{11}{e_6},[{e_2},{e_3}]=\frac{\sqrt {165}}{33}{e_6}.$&&&\\
\hline
13.3&-&-&-&$7$\\
\hline
14.1&$[{e_1},{e_2}]=\frac{\sqrt {55}}{22}{e_5},[{e_1},{e_3}]={\frac {3\sqrt {11}}{22}}{e_4},[{e_1},{e_4}]=\frac{\sqrt {22}}{11}{e_6} $&$\frac{5}{11}\diag(1,2,2,3,3,4)$&$ \frac{25}{22}$&$ 6$\\
\hline
14.2&$[{e_1},{e_2}]=\frac{\sqrt {55}}{22}{e_5},[{e_1},{e_3}]=-{\frac {3\sqrt {11}}{22}}{e_4},[{e_1},{e_4}]=  \frac{\sqrt {22}}{11}{e_6}$&$\frac{5}{11}\diag(1,2,2,3,3,4)$&$\frac{25}{22}$&$6$\\
\hline
14.3&$[{e_1},{e_2}]= \frac{\sqrt{6}}{6}{e_4},[{e_1},{e_3}]=\frac{\sqrt{6}}{6}{e_5},[{e_1},{e_4}]=\frac{\sqrt{6}}{6}{e_6}$&$\frac{1}{6}\diag(3,5,6,8,9,11)$&$\frac{7}{6}$&$7$\\
\hline
15.1&-&-&-&$5$\\
\hline
15.2&-&-&-&$5$\\
\hline
15.3&$ [{e_1},{e_2}]=\frac{\sqrt {21}}{14}{e_5},[{e_1},{e_5}]=\frac{\sqrt {42}}{14}{e_6},[{e_2},{e_3}]=\frac{\sqrt {35}}{14}{e_4}$&$\frac{5}{28}\diag(2,4,3,7,6,8)$&$ \frac{25}{28}$&$4$\\
\hline
21.1&
$[{e_1},{e_2}]=-\frac{\sqrt {66}}{44}{e_4} - \frac{\sqrt {22}}{44}{e_5},[{e_1},{e_3}]={\frac {3\sqrt {22}}{44}}{e_4} + \frac{\sqrt {66}}{44}{e_5},$&
\multirow{2}{*}{$\frac{5}{11}\diag(1,2,2,3,3,4)$}& \multirow{2}{*}{$ \frac{25}{22} $}&\multirow{2}{*}{$ 6$}\\
&$[{e_1},{e_4}]=-\frac{\sqrt {22}}{11}{e_6},[{e_2},{e_3}]=\frac{\sqrt {66}}{22}{e_6}$.&&&\\
\hline
21.2&$ [{e_1},{e_2}]= \frac{\sqrt {6}}{6}{e_3},[{e_1},{e_3}]= \frac{\sqrt {6}}{6}{e_6}, [{e_2},{e_4}]= \frac{\sqrt {6}}{6}{e_6}$&$ \frac{1}{6}\diag(3,5,8,6,9,11)$&$\frac{7}{6}$&$7$\\
\hline
21.3&$[{e_1},{e_2}]=\frac{\sqrt {6}}{6}{e_3},[{e_1},{e_3}]=-\frac{\sqrt {6}}{6}{e_6},[{e_2},{e_4}]=\frac{\sqrt {6}}{6}{e_6}$&$\frac{1}{6}\diag(3,5,8,6,9,11)$&$\frac{7}{6}$&$7$\\
\hline
22.1&$[{e_1},{e_2}]=\frac{1}{2}\,{e_5},[{e_1},{e_5}]=\frac{1}{2}\,{e_6}$&$\frac{1}{4}\diag(2,4,5,5,6,8)$&$\frac{5}{4}$&$8$\\
\hline
\end{tabular}
\end{center}
\caption{Classification of minimal compatible metrics on symplectic three-step Lie algebras of dimension $6$. \newline Here,
$f_{1}(\lambda)=\frac{\sqrt{2}}{4}{\frac { \sqrt {{\lambda}^{2}+\lambda+1}}{{\lambda}^{2}+\lambda+1}}$ and $f_{2}(\lambda)  = \frac{\sqrt {6}}{6} \frac{\sqrt { {\lambda}^{2}-\lambda+1 }}{ {\lambda}^{2}-\lambda+1 } $}\label{minmetricas3pasos}
\end{table}

\begin{table}[h!]
\begin{center}
\begin{tabular}{|c|l|}
\hline
\multicolumn{1}{ |c| }{\centering Not. } &
\multicolumn{1}{  c|    }{\centering Chern-Ricci operator}
\tabularnewline
\hline
\hline
10.1& Chern-Ricci flat\\
\hline
10.2& Chern-Ricci flat\\
\hline
11.1 $(\lambda=0)$& Chern-Ricci flat\\
\hline
11.2 $(\lambda=0)$& Chern-Ricci flat\\
\hline
12.1& $\chP_{\lambda}=\frac{\sqrt{2}}{8}\,{\left( 1+ \sign \left( \lambda \right)  \right)  \left| \lambda+1 \right|f_{1}(\lambda)} (E_{5,1}-E_{6,2})$\\
\hline
13.1& $\chP_{\lambda}=\frac{\sqrt{6}}{12}|\lambda-1|f_{2}(\lambda)(E_{5,1}-E_{6,2})$\\
\hline
13.2& {see Example \ref{ejemplofeo}} \\
\hline
13.2 $(\lambda=\frac{1}{4})$& $\chP = \frac{\sqrt{30}}{132}(E_{4,1}-E_{6,3})+\frac{\sqrt{2}}{22}(E_{5,1}-E_{6,2})$\\
\hline
14.1& $\chP = \frac{3\sqrt{2}}{44}(E_{4,1}-E_{6,3})$\\
\hline
14.2& $\chP = -\frac{3\sqrt{2}}{44}(E_{4,1}-E_{6,3})$\\
\hline
14.3& $\chP =  \frac{1}{12}(E_{5,1}-E_{6,2})$\\
\hline
15.3& $\chP = \frac{3\sqrt{2}}{56} ( E_{5,1}-E_{6,2})$\\
\hline
21.1& $ \chP =  -\frac{3}{44}(E_{4,1}-E_{6,3}) + \frac{\sqrt{3}}{44} (E_{5,1} - E_{6,2})$\\
\hline
21.2& $\chP =\frac{1}{12}(E_{5,1}-E_{6,2})$\\
\hline
21.3& $ \chP = -\frac{1}{12}(E_{5,1}-E_{6,2})$\\
\hline
22.1& $\chP = \frac{1}{8}(E_{5,1}-E_{6,2})$\\
\hline
\end{tabular}
\end{center}
\caption{Chern-Ricci operator of minimal compatible metrics on symplectic three-step Lie algebras of dimension $6$. \newline Here,
$f_{1}(\lambda)=\frac{\sqrt{2}}{4}{\frac { \sqrt {{\lambda}^{2}+\lambda+1}}{{\lambda}^{2}+\lambda+1}}$ and $f_{2}(\lambda)  = \frac{\sqrt {6}}{6} \frac{\sqrt { {\lambda}^{2}-\lambda+1 }}{ {\lambda}^{2}-\lambda+1 } $}\label{minmetricas3pasosCHP}
\end{table}

\subsection{Symplectic two-step nilpotent Lie algebras}
Here, we give the classification of minimal compatible metrics on two-step nilpotent Lie algebras,
which determines immediately a soliton almost K\"{a}hler structure, because the Chern-Ricci operator
is always zero.

\begin{example}
Consider the free $2$-step nilpotent Lie algebra of $\rank 3$; $\ngo_{18}:=[e_1,e_2] =e_4, [e_1,e_3] = e_5, [e_2,e_3] = e_6$. By \cite[Theorem 5]{KHAKIMDJANOV1}, $\ngo_{18}$ carries two curves of non-equivalent symplectic structures, namely $\omega_{1}(s)$ and $\omega_{2}(t)$ (with $s\in \RR\setminus\{0,1\}$ and $t \in \RR\setminus\{0\}$), and a isolated symplectic structure $\omega_{3}$. In this example we want to prove that every pair $(\ngo_{18},\omega_{i})$ admits a minimal compatible metric.

We take the case of the first curve
$$
\omega_{1}(s)= e_{1}^{\ast}\wedge e_6^{\ast} +se_2^{\ast}\wedge e_5^{\ast} +(s -1)e_3^{\ast}\wedge e_4^{\ast},\,\mbox{with }s\in\RR\setminus\{0,1\}
$$
and let $g=\diag(1,1,1,s-1,s,1)$. The change of basis given by $g$ defines a symplectomorphism from $(\ngo_{18},\omega_{1}(s))$ to $(\RR^6,\mu_{s},\omega_{cn})$ with
\begin{eqnarray*}
   \mu_{s} &=&  [e_1, e_2] = (s-1)e_4, [e_1, e_3] =s e_5, [e_2, e_3] =  e_6
\end{eqnarray*}
It is obvious that $(\RR^6,\mu_{s},\omega_{cn})$ is written in a $\omega_{cn}$-nice basis, because
$$
\mathfrak{R}_{\omega_{cn}}(\mu_{s})=\{\frac{1}{2}\diag(-1,-1,-1,1,1,1)\}
$$
and that from this $(\ngo_{18},\omega_{1}(s))$ admits a minimal compatible metric which can be found by solving the equation
$$
\mm_{\spg}(a\cdot\mu_{s})=\frac{1}{2}\diag(-1,-1,-1,1,1,1)
$$
for $a\in\exp(\ag_{\omega_{cn}})$.

Let $a=\exp(X)$ with
$$
X = \frac{1}{2}\diag(\ln  ( 4\,{s}^{2}-4\,s+4),0,0,0,0,-\ln  ( 4\,{s}^{2}-4\,s+4 )).
$$
The change of basis given by $a$ defines the curve
$$
\begin{array}{rcl}
\widetilde{\mu_{s}}  &:=& [{e_1},{e_2}]=\frac{1}{2}\, \left( s-1 \right) \sqrt { \left( {s}^{2}-s+1 \right) ^{-1}}{ e_4},
                      [{e_1},{e_3}]=\frac{1}{2}\,s\,\sqrt { \left( {s}^{2}-s+1 \right) ^{-1}}{e_5},\\
                 & &  [{e_2},{e_3}]=\frac{1}{2}\,\sqrt { \left( {s}^{2}-s+1 \right) ^{-1}}{e_6}
\end{array}
$$
and it is a solution of the above equation, since
$$
\mm_{\glg}(\widetilde{\mu_{s}}) = \frac{1}{2({s}^{2}-s+1)}\diag(-2\,{s}^{2}+2\,s-1,-{s}^{2}+2\,s-2,
-({s}^{2}+1), \left( s-1 \right) ^{2},{s}^{2},1),
$$
and thus
$$
\begin{array}{rcl}
\mm_{\spg}(\widetilde{\mu_{s}}) & = & \frac{1}{2}(\mm_{\glg}(\widetilde{\mu_{s}})+J.\mm_{\glg}(\widetilde{\mu_{s}}).J)\\
                            &&\\
                            & = & \frac{1}{2}\diag(-1,-1,-1,1,1,1)\\
                            &&\\
                            & = & -\frac{3}{2}\Id + \underbrace{\diag(1,1,1,2,2,2)}_{\mbox{Derivation}};\
\end{array}
$$
The canonical inner product of $\RR^6$ defines a minimal compatible metric on $(\RR^6,\widetilde{\mu}_{s},\omega_{cn})$.

Now, consider the case of
$$
\omega_{2}(t)= e_1^{\ast}\wedge e_5^{\ast} + t e_1^{\ast}\wedge e_6^{\ast} -t e_2^{\ast}\wedge e_5^{\ast} + e_2^{\ast}\wedge e_6^{\ast} - 2te_3^{\ast}\wedge e_4^{\ast}.
$$
From Example \ref{ejemploLibre}, we have $(\ngo_{18},\omega_{2})$ is equivalent to $(\RR^6,\mu_{t},\omega_{cn})$ with
\begin{eqnarray*}
 \mu_{t} &=& \left\{[{e_1},{e_2}]=-2\,t\,{e_4},[{e_1},{e_3}]=-
t\,{e_5}+{e_6},[{e_2},{e_3}]={e_5}+t\,{e_6}\right..
\end{eqnarray*}
It is easy to see that $(\RR^6,\mu_{t},\omega_{cn})$ is written in a $\omega_{cn}$-nice basis and that the set $\mathfrak{R}_{\omega_{cn}}(\mu_{t})$ is given by
$$
\begin{array}{l}
\mathfrak{R}_{\omega_{cn}}(\mu)=\{\diag(-1,0,-\frac{1}{2},\frac{1}{2},0,1),\diag(0,-1,-\frac{1}{2},\frac{1}{2},1,0),\\
\hspace{2cm}\frac{1}{2}\diag(-1,-1,-1,1,1,1)\}
\end{array}
$$

Like above, the Theorem \ref{nicegeneralizado} states that $(\ngo_{18},\omega_2(t))$ admits a minimal compatible metric and proceeding in a way similar, we find that $\widetilde{\mu_{t}}:=a\cdot\mu$ with
$$
a=\diag(1,1,2\sqrt{3t^2+1},\frac{1}{2\sqrt{3t^2+1}},1,1)
$$
is such that the canonical inner product of $\RR^6$ defines a minimal compatible metric on $(\RR^6,\widetilde{\mu_{t}},\omega_{cn})$ for any $t\in \RR\setminus\{0\}$.

In the last case,
$$
\omega_{3}:=e_3^{\ast} \wedge e_5^{\ast} - e_1^{\ast}\wedge e_6^{\ast} + e_2^{\ast} \wedge e_5^{\ast} + 2e_3^{\ast} \wedge e_4^{\ast}
$$
we can now proceed analogously like above. We leave it to the reader to verify that the following change of basis give a minimal compatible metric with $(\ngo_{18},\omega_{3})$:
\begin{eqnarray*}
g&:=&\diag(-1,\left(\begin{array}{cc}1&\frac{5}{6}\\1&-\frac{1}{6} \end{array}\right),\left(\begin{array}{cc}-2&-\frac{1}{6}\\2&\frac{7}{6} \end{array}\right),1),\\
a&:=&\diag(2\sqrt{3},1,1,1,1, \frac{\sqrt{3}}{6}).
\end{eqnarray*}
\end{example}

Proceeding in an entirely analogous way, we can study the remainder symplectic two-step Lie algebras given in \cite[Theorem 5. 24]{KHAKIMDJANOV1} and to obtain

\begin{theorem}\label{symtwo}
All symplectic two-step Lie algebras of dimension $6$ admit a minimal compatible metric, and, in consequence,
these admit a soliton almost K\"{a}hler structure.
\end{theorem}

\begin{remark}
We must say that we have found several mistakes in the classification given in \cite{KHAKIMDJANOV1}. For instance, 16.(b) does not define a symplectic structure, 9. is not a curve of non-equivalent symplectic structures. Some errors have already been corrected by personal communication with authors; as the symplectic structure given in 23.(c).
\end{remark}

In the Table \ref{minmetricas}, each Lie algebra defines a symplectic two-step Lie algebra given by $(\RR^6,\widetilde{\mu},\omega_{cn})$ and it is such that the canonical inner product on $\RR^6$ defines a minimal metric of scalar curvature equal to $-\frac{1}{4}$. In the column $||\beta||^2$ we give the norm squared of the stratum associated to the minimal metric and in Derivation column, we give the derivation of $(\RR^6,\widetilde{\mu})$ such that
$$
\mm_{\mathfrak{sp}_{6}(\RR)}(\widetilde{\mu})=-||\beta||^2\Id + \mbox{Derivation}.
$$
In the last column, We give the dimension of automorphism group of the symplectic two-step Lie algebra $(\RR^6,\widetilde{\mu},\omega_{cn})$.

\begin{table}[h]
\begin{center}
\begin{tabular}{|c|l|c|c|c|}
\hline
\multicolumn{1}{ |c| }{\centering Not. } &
\multicolumn{1}{  c|    }{\centering Critical point } &
\multicolumn{1}{  c|  } {\centering {Derivation}} &
\multicolumn{1}{  c|  }{\centering $||\beta||^2$ } &
\multicolumn{1}{  m{0.5cm }|  }{\centering $\dim$\\$\Aut$}
\tabularnewline
\hline
\hline
\multirow{2}{*}{16.1.}& $[{e_1},{e_2}]=\frac{\sqrt {2}}{4}{e_3}$,
        $[{e_1},{e_5}]=\frac{\sqrt {2}}{4}{e_6}$,
                      &\multirow{2}{*}{$\frac{1}{2}\diag(1,2,3,1,2,3)$}&\multirow{2}{*}{1}&\multirow{2}{*}{6}\\
&$[{e_2},{e_4}]=\frac{\sqrt {2}}{4}{e_6}$,                      $[{e_4},{e_5}]=\frac{\sqrt {2}}{4}{e_3}$&&&\\
\hline
\multirow{2}{*}{17}&$[{e_1},{e_3}]=\frac{\sqrt {6}}{6}{e_5}$,
                     $[{e_1},{e_4}]=\frac{\sqrt{6}}{6}{e_6}$,&\multirow{2}{*}{$\frac{1}{6}\diag(3,5,6,8,9,11)$}&\multirow{2}{*}{$\frac{7}{6}$}&\multirow{2}{*}{7}\\
&$[{e_2},{e_3}]=\frac{\sqrt {6}}{6}{e_6}$.&&& \\
\hline
\multirow{3}{*}{18.1.}&$[{e_1},{e_2}]= f_{1}(s)(( s-1 )\,{e_4})$, &\multirow{3}{*}{$\diag(1,1,1,2,2,2)$}&\multirow{2}{*}{$\frac{3}{2}$}&\multirow{2}{*}{$ 8$}\\
& $[{e_1},{e_3}]=f_{1}(s)(s\,{e_5})$,&&&\\
& $[{e_2},{e_3}]=f_{1}(s){e_6}$.&&&\\
\hline
\multirow{3}{*}{18.2.}&$[{ e_1},{e_2}]=f_{2}(t)(-2t\,{e_4}),$&\multirow{3}{*}{$\diag(1,1,1,2,2,2)$}&\multirow{3}{*}{$\frac{3}{2}$}&\multirow{3}{*}{$8$}\\
&$[{e_1},{e_3}]=f_{2}(t)(-t\,{e_5}+{e_6}),$&&&\\
&$[{e_2},{e_3}]=f_{2}(t)({e_5}+t\,{e_6})$&&&\\
\hline
\multirow{3}{*}{18.3.}&$[{e_1},{e_2}]=\frac{\sqrt {3}}{12}{e_4} - \frac{\sqrt {3}}{4}{e_5}$,&\multirow{3}{*}{$\diag(1,1,1,2,2,2)$}&\multirow{3}{*}{$\frac{3}{2}$}&\multirow{3}{*}{10}\\
&$[{e_1},{e_3}]=\frac{\sqrt {3}}{4}{e_4}-\frac{\sqrt {3}}{12}{e_5}$&&&\\
&$[{e_2},{e_3}]=-\frac{\sqrt {3}}{6}{e_6}$.&&&\\
\hline
23.1.&$[e_1, e_2] = \frac{1}{2}e_5$, $[e_1, e_3] = \frac{1}{2}e_6$&$ \frac{1}{4}\diag(4,5,6,8,9,10)$&$\frac{7}{4}$&$9$\\
\hline
23.2.&$[{e_1},{e_2}]=-\frac{1}{2}\,{e_4}$, $[{e_2},{e_3}]=\frac{1}{2}{e_6}$&$\diag(1,1,1,2,2,2)$&$\frac{3}{2}$&$8$\\
\hline
23.3.&$[e_1, e_2] =\frac{1}{2}e_5$, $[e_1, e_3] = -\frac{1}{2}e_4$&$\diag(1,1,1,2,2,2)$&$\frac{3}{2}$&$8$\\
\hline
24.1.&$ [e_1, e_4] = \frac{1}{2}e_6$, $[e_2, e_3] = \frac{1}{2}e_5$&$\frac{1}{2}\diag(1,1,2,2,3,3)$&$ 1$&$6$\\
\hline
24.2.&$[e_1, e_4] = \frac{1}{2}e_6$, $[e2, e3] = -\frac{1}{2}e_5$&$\frac{1}{2}\diag(1,1,2,2,3,3)$&$ 1$&$6$\\
\hline
25.&$ [e_1,e_2]=\frac{\sqrt{2}}{2}e_6$&$\frac{1}{2}\diag(3,4,5,5,6,7)$&$\frac{5}{2}$&$12$\\
\hline
\end{tabular}
\end{center}
\caption{Classification of minimal compatible metrics on symplectic two-step Lie algebras of dimension $6$. \newline Here,
$f_{1}(s)=\frac{1}{2}\sqrt {( {s}^{2}-s+1 )^{-1}}$ and $f_{2}(t)=\frac{1}{2}\sqrt { (3\,{t}^{2}+1 )^{-1}}$}\label{minmetricas}
\end{table}

\end{document}